\documentclass[a4paper,11pt,leqno]{amsart}
\usepackage{a4wide,color,array}
\usepackage{cite}
\usepackage{hyperref}
\usepackage{amsmath,amssymb,amsthm,color}
\hypersetup{linktocpage,colorlinks, citecolor={magenta}}
\usepackage[english]{babel}
\usepackage[utf8]{inputenc}

\usepackage{xspace}
\usepackage{bm}				% lettre mathématiques grasses et
\usepackage{bbm,upgreek}		% lettres math blackboard et grecques pour μm et pour β-decay
\usepackage[e]{esvect}		% des flèches de vecteurs plus élégantes
\usepackage{esint}			% intégrales diverses; installer esint-type1
\usepackage{etoolbox}			% nombreuses fonctions avancées indispensables
\usepackage{calc}				% calcul infix des longueurs
\usepackage{eso-pic}			% pour placer des élémnst das le background lors du shipout
\usepackage{datetime}			% formatage des datees, accès au temps

\usepackage{lmodern}
\numberwithin{equation}{section}

\usepackage{enumitem}
\setlist{nosep}
\setlist{noitemsep}

\newcommand{\R}{\mathbb{R}}

\newtheorem{theo}{Theorem}
\newtheorem{prop}{Proposition}[section]
\newtheorem{lem}[prop]{Lemma}
\newtheorem{coro}[prop]{Corollary}
\newtheorem{remark}[prop]{Remark}

\theoremstyle{plain}

\theoremstyle{definition}
\newtheorem{defi}[prop]{Definition}

\def \t0{\rightarrow 0} % Vers zéro
\def \be{\begin{equation}}
\def \ee{\end{equation}}

\def \hal{\frac{1}{2}}

\def \div{\mathrm{div} \,} % Divergence
\def \1{\mathbf{1}} % Fonction caractéristique

\def \ep{\varepsilon}

\def\a{\alpha}

\def\drd{\delta_{\R^{\d}}}

\def\nab{\nabla}

\def\({\left(}
\def\){\right)}

\def\XXint#1#2#3{{\setbox0=\hbox{$#1{#2#3}{\int}$}
     \vcenter{\hbox{$#2#3$}}\kern-.5\wd0}}

\def \XN{X_N}

\def \FN{F_N}

\def\g{\mathsf{g}}
\def\d{\mathsf{d}}
\def\F{\mathsf{F}}

\def\indic{\mathbf{1}}

\makeatletter
\def\namedlabel#1#2{\begingroup
    #2%
    \def\@currentlabel{#2}%
    \phantomsection\label{#1}\endgroup
}
\makeatother
\def \epsilon{\varepsilon}

\def \f{\mathsf{f}}

\def \veta{\vec{\eta}}

\def \rr{\mathsf{r}}

\def \vecr{\vec{\rr}}
\def\k{\mathsf{k}}
\def\s{\mathsf{s}}
\def\c{\mathsf{c}_{\d,\s}}
\def\cd{\mathsf{c}_\d}

\def\yg{|z|^\gamma}

\def\a{\alpha}

\def \f{\mathsf{f}}

\newcommand{\Hc}{\mathcal H}
\newcommand{\Pc}{\mathcal P}

\begin{document}
\title[Mean Field Limit for Coulomb Flows]{Mean Field Limit for Coulomb-Type Flows}
\author{Sylvia Serfaty, appendix with Mitia Duerinckx}
\date{}
\maketitle

\begin{abstract}
We establish the mean-field convergence for systems of points evolving along the gradient flow of their interaction energy when the interaction is the Coulomb potential or a super-coulombic Riesz potential,  for the first time in arbitrary dimension.  The proof is based on a modulated energy method using  a Coulomb or Riesz distance, assumes that the solutions of the limiting equation are regular enough and exploits a weak-strong stability property for them.
The method can handle the addition of a regular interaction kernel, and
applies also to  conservative and mixed   flows.
In the appendix, it is also adapted to prove the mean-field convergence of the solutions to Newton's law with Coulomb or Riesz interaction in the monokinetic case to  solutions  of an Euler-Poisson type system.
\end{abstract}

\section{Introduction}

The derivation of effective equations for  classical interacting many body systems is  an important question  in mathematical physics. Within it,  one of the most important problems, fundamental for plasma physics,  is the  derivation of the Vlasov-Poisson equation from Newton's law for $N$ particles i.e.~a system of the form
 \be\label{new0}
\left\{\begin{array}{l}
\dot x_i=v_i\\
\dot v_i= \displaystyle\frac{1}{N}\sum_{i\neq j}K(x_i,x_j) , \quad i=1, \dots, N\\ [3mm]
x_i(0)=x_i^0\quad v_i(0)= v_i^0\end{array}\right.\ee
where the pair interaction $K$ derives from the Coulomb potential, and 
 is still open in its full generality. 
 The large   $N $ or mean-field limit of first order systems of the form 
 \be\label{new0}
\left\{\begin{array}{l}
\dot x_i= \displaystyle\frac{1}{N}\sum_{i\neq j}K(x_i,x_j) , \quad i=1, \dots, N\\ [3mm]
x_i(0)=x_i^0 \end{array}\right.\ee
 is also a natural and interesting question.  Such systems can model interacting particles in physics (for instance the point vortex system in two-dimensional fluids), from a numerical  point of view they  correspond to particle approximations of PDEs, or in the case of gradient flows can serve to approximate  their equilibrium states. Motivations extend to biological and sociological sciences, including phenomena of flocking, swarming and aggregation
 (see for instance \cite{choicarrillohauray}) and the analysis of large neural networks in biology \cite{bft} and in machine learning  \cite{mmn,rve,bc}.  
 
In the above problems, the cases where $K$ is regular are well understood, while in contrast those where $K$ is singular are the most difficult and least understood. In this paper we will be particularly focusing on the case where $K$ corresponds to the Coulomb interaction (and some generalizations), arguably the most important case for physics. 
For  further details on the general mathematical aspects of mean-field limits, we refer to the reviews \cite{golse,jabin}.

 \smallskip

The rest of the introduction is structured as follows: in Section \ref{1.1} we introduce the exact equations that we will study as well as their limiting effective evolution equations, and describe the state of the art on such questions.
In Section \ref{1.2} we state the main result, in Section \ref{1.3} we comment on the assumptions, and in Section \ref{method} we give an extended proof outline for the Coulomb case.
Finally, Section \ref{introsupp} is an added section explaining how to treat the case of  gradient flow evolutions with thermal noise.

 \subsection{Problem and background}\label{1.1}
In this paper we consider  specifically systems with Coulomb, logarithmic or Riesz interaction with kernel deriving from
 \be \label{formeg}
\g(x)=|x|^{-\s} \quad \max(\d-2,0) \le \s<\d \qquad \text{for any } \d\ge 1 
\ee
or 
\be \label{glog}
\g(x)= -\log |x|\qquad \text{for }  \d=1 \ \text{or} \ 2 ,\ee where $\d$ is the dimension. 
In the case \eqref{formeg} with  $\s=\d-2$  and $\d\ge 3 $, or \eqref{glog} and $\d=2$, $\g$ is  exactly (a multiple of) the Coulomb kernel. In the other cases of \eqref{formeg} it is called a Riesz kernel. In some cases, we may add to the interaction force a regular part denoted $\F$.

We will consider first order dynamics 
 of the form  
\be \label{mfgf}
\left\{\begin{array}{l}
\dot x_i=-\displaystyle\frac{1}{N} \(\nab_{x_i}\mathcal H_N(x_1, \dots, x_N)+\sum_{j\neq i} \F(x_i-x_j)\), \quad i=1, \dots, N\\ [3mm]
x_i(0)=x_i^0\end{array}\right.\ee
or conservative evolutions of the form 
\be \label{cons}
\left\{ \begin{array}{l}
\dot x_i= \displaystyle - \frac{1}{N} \mathbb{J}\nab_{x_i}\mathcal H_N (x_1, \dots, x_N),\quad i=1, \dots, N\\ [3mm]
x_i(0)=x_i^0\end{array}\right.\ee
where $\mathbb{J}$ is an antisymmetric matrix, the points $x_i$ evolve in the whole space $\R^{\d}$ and their energy $\mathcal H_N$ is given by 
\be\label{1.3} \mathcal H_N(x_1, \dots, x_N)= \sum_{i\neq j} \g(x_i-x_j),\ee with $\g$ as above. The map $\F:\R^\d \to \R^\d$ is the additional interaction force that we can add in the dissipative case to illustrate the robustness of the method, we will require it to enjoy some Sobolev and H\"older  regularity.

Mixed flows of the form 
\be \label{mix}\dot x_i= - \frac{1}{N} \(\alpha I + \beta \mathbb{J}\) \(
\nab_{x_i}\mathcal H_N (x_1, \dots, x_N)+\sum_{j\neq i} \F(x_i-x_j)\) \qquad \alpha>0 \ee
can be treated with exactly the same proof, so can  the same dynamics \eqref{mfgf}, \eqref{cons} or \eqref{mix} with an additional forcing $\frac{1}{N}\sum_{i=1}^N \mathsf{V}(x_i) $ with $\mathsf{V}$ Lipschitz. These generalizations are left  to the reader, see in particular Remark \ref{remV}.

Studying the same evolutions with added noise 
\be \label{mfgfnoise}
d x_i=-\displaystyle\frac{1}{N} \(\nab_{x_i}\mathcal H_N(x_1, \dots, x_N)+\sum_{j\neq i} \F(x_i-x_j)\) dt + \sqrt{2 \theta} dW_i, \ee or 
\be \label{consnoise}
d x_i=-\displaystyle\frac{1}{N} \mathbb{J} \nab_{x_i}\mathcal H_N(x_1, \dots, x_N)dt + \sqrt{2 \theta} dW_i, \ee
with $dW_i$ being $N$ independent Brownian motions and $\theta \ge 0$ a temperature,
 is also very interesting and has motivations from Random Matrix Theory, it  is done in particular in  \cite{jabinwang,jabinwang2,bo,fhm,bcf,lly} (see also references therein). We will comment further on this at the end of the introduction in Section \ref{introsupp}.

In the appendix, we  consider the second order system corresponding to Newton's law for the energy $\mathcal H_N$
\begin{gather}\label{eq:MFL00}
\begin{cases}
\dot x_{i}=v_{i},\\
\dot v_{i}=-\frac1N\nabla_{x_i}\Hc_N(x_1,\ldots,x_N),\\
x_{i}(0)=x_{i}^0,\quad v_i(0)=v_{i}^0,
\end{cases}\qquad i=1,\ldots,N
\end{gather}
in the so-called monokinetic case. In the Coulomb case it is the true physics model for plasmas.

\smallskip

Consider the {\it empirical measure} 
\be\label{munt}
\mu_N^t:= \frac1N \sum_{i=1}^N \delta_{x_i^t}\ee  associated to a solution $X_N^t:=(x_1^t, \dots, x_N^t)$ of the flow \eqref{mfgf} or \eqref{cons}. If the points $x_i^0$, which themselves depend on $N$, are such that $\mu_N^0$   converges  to some regular measure $\mu^0$, then a formal derivation leads to expecting that for $t>0$, $\mu_N^t$ converges to the solution of the Cauchy problem with initial data $\mu^0$ for the limiting evolution equation 
\be \label{limeq}
\partial_t \mu= \div ((\nab \g+\F)*\mu) \mu)\ee
in the dissipative case \eqref{mfgf} or 
\be\label{limecons}
\partial_t \mu= \div (\mathbb{J}\nab (\g*\mu)\mu)\ee
in the conservative case \eqref{cons}, with $*$ denoting the usual convolution.

These equations should be understood in a weak sense. Equation \eqref{limeq} (with $\F=0$) is sometimes called the {\it fractional porous medium equation}.   The two-dimensional Coulomb version also arises as a model for the dynamics of vortices in superconductors. The construction of solutions, their regularity and basic properties,  are addressed in \cite{linz,duz,mz,as,sv} for the Coulomb case of \eqref{limeq}, in \cite{csv,cv,xz} for the case $\d-2<\s<\d$ of \eqref{limeq}, and \cite{delort,yudo,schochet2}  for the two-dimensional Coulomb case of   \eqref{limecons}.
\smallskip 

Establishing the convergence of the empirical measures to solutions of the limiting equations  is nontrivial because of the nonlinear terms in the equation and the singularity of the interaction $\g$.
In fact, because of the strength of the  singularity,  treating the case of  Coulomb interactions in dimension $\d \ge 3$ (and even more so that of  super-coulombic interactions) had remained an  open question for a long time, see  for instance   the introduction of \cite{jabinwang2} and the review \cite{jabin}. It was not even completely clear that the result was true without expressing it in some statistical sense (with respect to the initial data).

In \cite{jabinwang,jabinwang2}, Jabin and Wang introduced a new approach for the related problem of the mean-field convergence of the solutions of Newton's second order system of ODEs to the Vlasov equation, which allowed them to treat all interactions kernels with bounded gradients, but still not Coulomb interaction.  The same problem has been addressed in  \cite{Boers-Pickl-16,lp,Lazarovici-16} with  results that still require a cutoff of the Coulomb interaction.  Our method already allows to unlock the case of Coulomb interaction for  monokinetic data, this is the topic of Appendix~\ref{app1}. The non-monokinetic case is even much more challenging and remains open.
 \medskip

The previously known results on the problems we are addressing were  the following:
\begin{itemize}
\item In dimension 2, choosing  \eqref{glog} and $\mathbb{J}$ the rotation by $\pi /2$ in  \eqref{cons}  corresponds to the so-called {\it point vortex system} which is well-known in fluid mechanics (cf. for instance  \cite{mp}), and its mean-field convergence to the Euler equation in vorticity form \eqref{limecons}  was already established \cite{schochet}.  His proof can be readapted to treat  \eqref{mfgf} as well in that case.  Results of similar nature were also obtained in \cite{ghl}.
\item Hauray \cite{hauray} (see also \cite{choicarrillohauray}) treated the case of all sub-Coulombic interactions ($\s<\d-2$) for  (a possibly higher-dimensional generalization of) \eqref{cons}, where particles  can have positive and negative charges and thus can attract as well as repel. The proof, which relies on  the stability in $\infty$-Wasserstein distance of the limiting solution, cannot be adapted to $\s\ge \d-2$.
\item In dimension 1 and in the dissipative case \eqref{mfgf}, Carrillo-Ferreira-Precioso and Berman-Onnheim \cite{cfp,bo} proved the unconditional convergence for all $0<\s<1$ using the framework of Wasserstein gradient flows but their method, based on the convexity of the interaction in dimension 1, does not extend to higher dimensions.
\item Duerinckx \cite{du}, inspired by the {\it modulated energy} method of \cite{serf} for Ginzburg-Landau equations (where vortices also interact like Coulomb particles in dimension 2), was able to prove the result in the dissipative case \eqref{mfgf} for $\d=1$ and $\d=2$ with $\s<1$, conditional to the regularity of the limiting solution, as we have here.  
\end{itemize}
\smallskip 

In this paper, we prove the mean field convergence for \eqref{mfgf} and \eqref{cons} in all the cases \eqref{formeg} and \eqref{glog} in every dimension. This extends Duerinckx's result, which involves overcoming serious difficulties, as described further below, and we add the possible additional interaction force $\F$ in the dissipative case. We are limited to $\s<\d$ and this is natural since for $\s\ge\d$ the interaction kernel $\g$ is no longer integrable and the limiting equation is expected to be a different one.

As in \cite{du}, our proof is a  modulated energy argument inspired from \cite{serf}, which is a way of exploiting a weak-strong uniqueness principle for the limiting equation. As mentioned above, looking for a stability principle in some Wasserstein distance fails for the Coulomb singularity. Instead we use a distance which is built as  a Coulomb (or Riesz) metric, associated to the norm 
\be \label{coulomb}\|\mu\|^2=\iint \g(x-y)d\mu(x)d\mu(y).\ee
We are able to  show by a Gronwall argument on this metric that the equations \eqref{limeq} and \eqref{limecons} satisfy a weak-strong uniqueness principle, and this can be translated into a proof of stability and convergence to $0$ of the norm of $\mu_N^t-\mu^t$  (if it is initially small, it remains small for all  further times).

The proof is self-contained and  quantitative. It does not require   understanding any qualitative property of the trajectories of the particles, such as for instance their minimal distances along the flow.

After this work was completed, Bresch, Jabin and Wang \cite{bjw} were able to beautifully incorporate our modulated energy into the relative entropy method of  \cite{jabinwang2}, turning it into a {\it modulated free energy}, which is a physically very natural quantity. It allowed them to extend the result of \cite{jabinwang2} to more singular interactions, including Coulomb. Seen from our point of view, it allows
 to  treat the cases with added noise \eqref{mfgfnoise} (but unfortunately not \eqref{consnoise}). This can be explained very succintly, we do it in Section \ref{introsupp} for the reader's convenience.
\subsection{Main result}\label{1.2}
Let $X_N$ denote $(x_1, \dots, x_N)$ and let us define for any  probability measure $\mu$,
\begin{equation}\label{energydef}
F_N(X_N, \mu)  = \iint_{\R^{\d}\times \R^{\d}\setminus \triangle} \g(x-y)d \Big( \sum_{i=1}^N \delta_{x_i}-N \mu\Big) (x)d \Big( \sum_{i=1}^N \delta_{x_i}-N\mu\Big) (y)
\end{equation}
where $\triangle$ denotes the diagonal in $\R^{\d} \times \R^{\d}$.  We choose for ``modulated energy"
$$F_N(\XN^t, \mu^t)$$
where $X_N^t=(x_1^t, \dots, x_N^t)$ are the solutions to \eqref{mfgf} or \eqref{cons}, 
 and $\mu^t$ solves the expected limiting PDE. 
The function  $F_N$ is not positive, however it is bounded below (by $-CN^{1+\frac\s\d}$ in the case \eqref{formeg}, respectively $ -\(\frac{N}{\d}\log N\)- C N$ in the case \eqref{glog}, 
 see Corollary \ref{corominob}). It  turns out that  also metrizes  at least weak convergence, as described in Proposition~\ref{procoer}. One may thus think of it as
 a good notion of distance from $\mu_N^t $ to $\mu^t$, more precisely $\frac{1}{N^2} F_N(X_N^t, \mu^t)$ is a good distance.

Our main result is a Gronwall inequality  on the time-derivative of $F_N(\XN^t, \mu^t)$, which implies     a quantitative rate of convergence of $\mu_N^t $ to $\mu^t$ in that metric.
 
Throughout the paper,  $(\cdot)_+$ denotes the positive part of a number.  The parameter $\s$ refers to the exponent in \eqref{formeg} while in the logarithmic case \eqref{glog} it is  set   to be $0$.  The space $\dot{H}^m(\R^\d)$ is the homogeneous Sobolev space of functions $u$ whose Fourier transform $\hat u$ satisfies $|\xi|^m \hat u(\xi)\in L^2(\R^\d)$.
 
\begin{theo}\label{th1}
Assume that $\g$ is of the form  \eqref{formeg}  or \eqref{glog}. 
Assume that 
$ \F\in \dot{H}^{\frac{\d-\s}{2}} (\R^\d)\cap C^{0,\alpha}(\R^\d)$ for some $\alpha>0$ and $\nab \F \in L^q(\R^\d)$ for some $1\le q\le \infty$.
Assume \eqref{limeq}, resp. \eqref{limecons}, admits a solution  $\mu^t$ such that, for some $T>0$, 
\be\label{nmut}
\begin{cases} \mu^t \in L^\infty([0,T], L^\infty (\R^{\d})),  \ \text {and } \nab^2 \g* \mu^t \in L^\infty  ([0,T], L^\infty (\R^{\d})) \quad \text{if } \ \s<\d-1
\\\mu^t \in L^\infty([0,T], C^\sigma (\R^{\d}))  \ \text {with }  \sigma > \s-\d+1,   \  \text{and } \nab^2 \g* \mu^t \in L^\infty  ([0,T], L^\infty (\R^{\d})) \  \text{if } \ \s\ge \d-1.
\end{cases}
\ee
 Let $X_N^t $ solve \eqref{mfgf}, respectively \eqref{cons}. Then  there exist positive constants  $C_1, C_2$ depending only on the norms of $\mu^t$ controlled by \eqref{nmut}  and those of $\F$, and an exponent $\beta<2$ depending only on $\d,\s,\alpha,\sigma$, such that 
for every $t\in [0, T]$ we have 
\be\label{distcoul}
F_N(X_N^t , \mu^t)
 \le \( F_N(X_N^0, \mu^0)+ C_1 N^\beta\) e^{C_2 t} .\ee
In particular, using the notation \eqref{munt},  if $\mu_N^0 \rightharpoonup \mu^0$ and is such that 
$$\lim_{N\to \infty}\frac{1}{N^2} F_N(X_N^0, \mu^0) =0,$$
then the same is true for every $t \in [0, T]$ and 
\be \mu_N^t \rightharpoonup \mu^t\ee in the weak sense.
\end{theo}
Establishing the convergence of the empirical measures 
is essentially equivalent to proving {\it propagation of molecular chaos} (see \cite{golse,hm,jabin} and references therein) which means showing that if $f_N^0(x_1, \dots, x_N)$ is the initial probability density of   seeing particles at $(x_1, \dots, x_N)$ and if $f_N^0$ converges to some factorized state  $\mu^0 \otimes \dots \otimes \mu^0$, then  the $k$-point marginals $f_{N,k}^t$ converge for all time to $(\mu^t )^{\otimes k}$. With Remark \ref{remchaos}, our result implies a convergence of this type as well.

Let us point out that using a Fourier-based point of view on \eqref{energydef} Bresch-Jabin-Wang were able (see \cite{bjw,bjw2}) to subsequently relax  the assumptions on the interaction $\g$: it does not need to be Coulomb or Riesz (a bit like with the added regular force in \eqref{mfgf}) but may contain a mildly singular attractive part, as long as a sufficiently strong repulsive part is  still present.

\subsection{Comments on the assumptions}\label{1.3}Let us now comment on the regularity assumption made in \eqref{nmut}.
First of all, one can check (see Lemma \ref{lemregu}) that the assumption \eqref{nmut} is implied by 
\be \label{assumfaible}\mu \in L^\infty([0,T], C^\theta (\R^{\d}) ) \quad \text{for some} \ \theta>\s-\d+2,
\ee which coincides with the assumption made in \cite{du}
and is a bit stronger.
This weakening of the assumption allows to include for instance the case of measures which are (a regular function times) the characteristic function of some regular set, such as in the situation of vortex patches for the Euler equation in vorticity form,  corresponding to \eqref{cons} in the two-dimensional logarithmic case.  These vortex patches were first studied in \cite{chemin,bc,serfati} where it was shown that if the patch initially has a $C^{1,\alpha}$ boundary  this remains the case over time, and our second assumption that  the velocity  $\nab \g* \mu^t$ be Lipschitz holds as well (see also \cite{bk}). It is not too difficult to check that in all dimensions this second condition holds any time $\mu$ is $C^\sigma$ with $\sigma>0$ away from a finite number of $C^{1,\alpha}$ hypersurfaces.
More generally, such situations with patches can be expected to naturally arise in all the Coulomb cases. For instance, in the dissipative Coulomb case \eqref{mfgf} with $\F=0$, in any dimension
a  self-similar solution in the form of (a constant multiple of) the characteristic function of an expanding ball was exhibited in \cite{sv} and  shown to be an attractor of the dynamics. 
For the non-Coulomb dissipative cases, the corresponding self-similar solutions, called {\it Barenblatt solutions}, are of the form 
$$t^{-\frac{\d}{2+\s}}(a-b x^2 t^{-\frac{2}{2+\s}})_+^{\frac{\s- \d+ 2}{2}}$$ 
as shown in \cite{bik,cv} (and this formula retrieves the solution of \cite{sv} for $\s=\d-2$).

 If the initial $\mu^0$ is sufficiently regular, the stronger assumption \eqref{assumfaible} is known to hold  with $T=\infty$ for the Coulomb case (see \cite{linz} where the proof works as well in higher dimensions), and it is known up to some $T>0$ in the case $(\d-2)_+<\s\le \d-1$ \cite{xz}. 
As for \eqref{limecons}, to our knowledge the desired regularity is only known in dimension 2 for the Euler equation in vorticity form  (see \cite{volib,chemin}), although the arguments of \cite{xz} written for the dissipative case seem to also apply to the conservative one. Our convergence result thus holds in these cases, under the  assumption  that the limit $\mu^0$ of $\mu_N^0 $ is sufficiently regular and that $F_N(\XN^0, \mu^0) =o(N^2)$. Note that, as shown in \cite{du}, the latter  is implied by the convergence of the initial energy 
$$\lim_{N\to \infty} \frac{1}{N^2 } \sum_{i\neq j} \g(x_i^0-x_j^0) = \iint_{\R^{\d}\times \R^{\d}} \g(x-y) d\mu^0(x)d \mu^0(y)$$ which can be viewed as a well-preparedness condition.
For any regular enough $\mu^0$, one may for instance build  initial conditions satisfying this assumption (and something even stronger) by the construction in \cite{ps}.

For $\d-1<\s<\d$, even the local in time propagation of  regularity of solutions of \eqref{limeq} remains an open problem. Note that the uniqueness of regular enough solutions  is always implied by the weak-strong stability argument we use,  reproduced in Section \ref{method} below.

Requiring some  regularity of the solutions to the limiting equation for establishing convergence  with relative entropy / modulated entropy / modulated energy methods is fairly  common: the same situation appears for instance in \cite{jabinwang,jabinwang2} or in the derivation of the Euler equations from the Boltzmann equation via the modulated entropy method, see  \cite{lsr} and references therein.

\subsection{The method}\label{method}
As mentioned, our method exploits a 
weak-strong uniqueness principle for the solutions of \eqref{limeq}, resp. \eqref{limecons}, which is exactly the same as \cite[Lemma 2.1, Lemma 2.2]{du} (and can be easily readapted to the conservative case) and states that if $\mu_1^t $ and $\mu_2^t$ are two $L^\infty$ solutions to \eqref{limeq} such that $\nab^2 (\g* \mu_2)\in L^1([0,T], L^\infty)$, we have 
\begin{multline}\label{wsu}\iint_{\R^{\d}\times \R^{\d}} \g(x-y) d(\mu_1^t-\mu_2^t)(x) d(\mu_1^t-\mu_2^t)(y)\\ \le e^{C\int_0^t \|\nab^2 (\g* \mu_2^s)\|ds}\iint_{\R^{\d}\times \R^{\d}} 
\g(x-y) d(\mu_1^0-\mu_2^0)(x) d(\mu_1^0-\mu_2^0)(y).\end{multline}
But the Coulomb or Riesz energy \eqref{coulomb}
is nothing else than the fractional Sobolev $H^{-\a}$ norm of $\mu$ with $\a= \frac{\d-\s}{2}$, hence this  is a good metric of convergence and implies the weak-strong uniqueness property.

A crucial ingredient is  the use of the stress-energy (or energy-momentum) tensor which naturally appears  when taking the {\it  inner variations} of the energy \eqref{coulomb}, i.e. computing $\frac{d}{dt}|_{t=0}  \|\mu \circ (I+ t \psi)\|^2$  (this is standard in the calculus of variations, see for instance \cite[Sec. 1.3.2]{helein}).
To explain further, let us restrict for now to the Coulomb case, and set $h^\mu= \g* \mu$, the Coulomb potential generated by $\mu$.  In that case, we have 
\be \label{eqdelta}-\Delta h^\mu= \cd \mu\ee
for some constant $\cd$ depending only on $\d$.
The first key is to reexpress the Coulomb energy \eqref{coulomb}  as a single integral in $h^\mu$, more precisely we easily find via an integration by parts that if $\mu$ is a measure with $\int d\mu=0$, 
\be\label{intbp}
\iint_{\R^{\d}\times \R^{\d}} \g(x-y)d\mu(x)d \mu(y)= \int_{\R^{\d}} h^\mu d\mu= -\frac{1}{\cd}\int_{\R^{\d}}h^\mu \Delta h^\mu= \frac{1}{\cd}\int_{\R^{\d}} |\nab h^\mu|^2.\ee The stress-energy tensor
is then defined as the $\d\times \d$ tensor with coefficients
\be [h^\mu, h^\mu]_{ij}= 2 \partial_i h^\mu \partial_j h^\mu - |\nab h^\mu|^2 \delta_{ij},\ee
with $\delta_{ij} $ the Kronecker symbol.
We may compute that if $\mu$ is regular enough
\be \label{divstres}\div [h^\mu, h^\mu]= 2 \Delta h^\mu\nab h^\mu=   - 2 \cd \mu \nab h^\mu.\ee (Here the divergence is a vector with entries equal to the divergence of each row of $[h^\mu, h^\mu]$.)
We thus see how this stress-energy tensor allows to give a weak meaning  to the product $\mu \nab h^\mu= \mu \nab \g * \mu$, with $[h^\mu, h^\mu]$ well-defined in energy space and pointwise controlled by $|\nab h^\mu|^2$, which can by the way serve to give a notion of weak solutions of the equation in the energy space (as in \cite{delort,linz}).
Note that in dimension $2$, it is known since \cite{delort} that even though $[h^\mu, h^\mu]$ is nonlinear, it is stable under weak limits in energy space provided $\mu$ has a sign, but this fact {\it does not extend} to higher dimension.
\smallskip

Let us now present the short proof of \eqref{wsu} as it will be a model for the main proof. We focus on the dissipative case (the conservative one is an easy adaptation) and still the Coulomb case for simplicity with no additional interaction $\F$ (when present, the additional terms can be absorbed thanks to the dissipation).
Let $\mu_1$ and $\mu_2$ be two solutions to \eqref{limeq} and $h_i=\g* \mu_i$ the associated potentials, which solve
 \eqref{eqdelta}. Let us compute 
\begin{multline}\label{calc}
\partial_t \int_{\R^{\d}} |\nab (h_1-h_2)|^2= 
 2 \cd \int_{\R^{\d}} (h_1-h_2)\partial_t (\mu_1-\mu_2) \\
=  2 \cd\int_{\R^{\d}} (h_1-h_2) \div(\mu_1 \nab h_1 - \mu_2 \nab h_2)\\
= -2 \cd \int_{\R^{\d}} (\nab h_1- \nab h_2)\cdot (\mu_1 \nab h_1 -\mu_2 \nab h_2)\\
= - 2 \cd \int_{\R^{\d}}|\nab (h_1-h_2)|^2 \mu_1 - 2 \cd \int_{\R^{\d}} \nab h_2\cdot \nab (h_1-h_2) (\mu_1-\mu_2)\end{multline}
In the right-hand side, we recognize from \eqref{divstres} the divergence of the stress-energy tensor $[h_1-h_2, h_1-h_2]$, hence 
\begin{equation*}
\partial_t \int_{\R^{\d}} |\nab (h_1-h_2)|^2\le \int_{\R^{\d}} \nab h_2 \cdot \div [h_1-h_2, h_1-h_2]\end{equation*}
so if $\nab^2 h_2$ is bounded, we may integrate by parts the right-hand side and bound it pointwise by
$$\|\nab ^2 h_2\|_{L^\infty} \int_{\R^{\d}} \left| [ h_1-h_2, h_1-h_2]\right| \le 2 \|\nab ^2 h_2\|_{L^\infty} \int_{\R^{\d}} |\nab (h_1-h_2)|^2, $$ and the claimed result follows by Gronwall's lemma.

In the Riesz case, the Riesz potential $h^\mu=\g*\mu$ is no longer the solution to a local equation, and to find a replacement to \eqref{eqdelta}--\eqref{divstres}, we use an extension procedure as popularized by \cite{caffsilvestre} in order to obtain a local integral in $h^\mu$ in the extended space $\R^{\d+1}$.

In the discrete case of the original ODE system, all the above integrals are singular and this constitutes the main difficulty  overcome in this paper.
%and these singularities have to be removed. 
 In place of the second term in the right-hand side of \eqref{calc},  we then  have to control a  term which by symmetry can be written in the form 
\be\label{pt}\iint_{\R^{\d}\times \R^{\d} \backslash \triangle}( \nab h^{\mu_1}(x)-\nab h^{\mu_2}(y)) \cdot \nab \g(x-y) d(\mu_1-\mu_2)(x)d(\mu_1-\mu_2)(y)\ee
where $\triangle $ denotes the diagonal, $\mu_1$ is the limiting measure $\mu^t$ and $\mu_2$ is the discrete empirical measure $\mu_N^t$.
Such terms are well known (see for instance \cite{schochet}), and create
the main difficulty due to the singularity of $\g$. When removing the diagonal,  the positivity manifested by \eqref{intbp}  is in effect lost. 
We are however able to prove the following crucial functional inequality.
\begin{prop}\label{32} Assume that $\mu$ is a probability density, with  
$\mu \in C^{\sigma}(\R^{\d})$ with $\sigma>\s-\d+1$  if $\s \ge \d-1$; respectively 
$\mu \in L^\infty(\R^{\d}) $ or $\mu \in C^\sigma(\R^{\d})$ with $\sigma>0$  if $\s <\d-1$.
For any $X_N \in (\R^\d)^N$ and any Lipschitz map
$\psi: \R^{\d}\to \R^{\d}$, we have 
\begin{multline}\label{crucial}
\left|\iint_{\triangle^c} \( \psi(x)- \psi(y)\)\cdot  \nab \g(x-y) d (\sum_{i=1}^N \delta_{x_i}  - N\mu)(x) d(\sum_{i=1}^N \delta_{x_i} - N \mu)(y)\right| 
\\ \le  
C\|\nab\psi\|_{L^\infty}\( F_N(\XN, \mu)+  (1+\|\mu\|_{L^\infty} )  N^{2-\frac{\d-\s}{\d(\d+1)}} +2 \(\frac{N}{\d}\log N \)\indic_{\eqref{glog}}\)  \\
+C \min \(   \|\psi\|_{L^\infty} \|\mu\|_{L^\infty}  N^{1+\frac{\s+1}{\d}} +  \|\nab \psi\|_{L^\infty}\|\mu\|_{L^\infty}N^{1+\frac{\s}{\d}},
  \|\psi\|_{W^{1,\infty} } \|\mu\|_{C^{\sigma}} N^{1+\frac{\s +1-\sigma}{\d} }\)
\\ + C\begin{cases}
   \|\nab \psi\|_{L^\infty}(1+ \|\mu\|_{C^{\sigma}}) N^{2-\frac{1}{\d}} \qquad \text{if} \  \s \ge \d-1\\
  \|\nab \psi\|_{L^\infty}(1+ \|\mu\|_{L^\infty}) N^{2-\frac{1}{\d}}
 \qquad \text{if} \  \s < \d-1,
  \end{cases}
\end{multline} where $C$ depends only on $\s$, $\d$.
 \end{prop}
 The right-hand side should be read as $C_{\psi, \mu, s, \d}( F_N(X_N, \mu)+ N^\beta)$ for some $\beta <2$. 
This inequality is the main novelty of the paper.  Even though the term $(\psi(x)-\psi(y))
\cdot \nab \g(x-y)$ has a singularity of same order as $\g(x-y)$ near the diagonal, the inequality is not at all obvious due to the lack of positivity of the integrand and its proof is rendered difficult by the handling of the removed diagonal terms.
Note that in \cite{gp} Golse and Paul were able to treat the mean field limit for the {\it quantum} Coulomb dynamics, relying on this inequality.

To give more insight into the proof of this proposition, we need to discuss the electric representation of the modulated energy $F_N$, again restricting to the Coulomb case. For that we introduce the electric potential
$$H_N^{\mu}[X_N]=\g* \( \sum_{i=1}^N \delta_{x_i}-N \mu\).$$
Arguing as in \eqref{eqdelta}--\eqref{intbp} we would like to rewrite 
$F_N(X_N)$ as $\int_{\R^\d} |\nab H^\mu_N[X_N]|^2$. This is not quite correct due to the divergence of $H_N^\mu$ at the points $x_i$. This can however be corrected by using mollified Dirac masses and setting for any $\vec{\eta}=(\eta_1, \dots , \eta_N)\in \R^N$, 
$$H_{N, \vec{\eta}}^{\mu}=\g* \( \sum_{i=1}^N \delta_{x_i}^{(\eta_i)}-N \mu\)$$
where we dropped the $[X_N]$ in the notation and let $\delta_x^{(\eta)}$ denote the uniform measure of mass $1$ on $\partial B(x, \eta)$. This effectively truncates $H_N^{\mu}$ at scale $\eta_i$ around each $x_i$, i.e. a scale which depends on each point.
Reinserting the diagonal terms, it is then not too difficult to check  that 
\begin{align*}F_N(X_N, \mu)  = &  \lim_{\eta_i\to 0} 
\iint \g(x-y) d\( \sum_{i=1}^N \delta_{x_i}^{(\eta_i)}-N \mu\)(x)d\( \sum_{i=1}^N \delta_{x_i}^{(\eta_i)}-N \mu\)(y)
\\ & - \sum_{i=1}^N \iint \g(x-y) d \delta_{x_i}^{(\eta_i)}(x)d\delta_{x_i}^{(\eta_i)}(y)
\\
= &\frac{1}{\cd}\lim_{\eta_i\to 0} \(
\int_{\R^\d} |\nab H_{N, \veta}^\mu|^2
-\cd \sum_{i=1}^N   \g(\eta_i)\).\end{align*} This effectively gives a renormalized meaning to the identity \eqref{intbp} in this setting.

The idea of expressing the interaction energy as a local integral in $H_N^\mu$ and its renormalization procedure were previously used in the study of Coulomb and Riesz energies in  \cite{rs,ps,ls1,ls2}, but it was not clear how to adapt these ideas to control \eqref{pt}: in fact in \cite{du} this was dealt with by a ball-construction procedure inspired from the analysis of Ginzburg-Landau vortices, which led to nonoptimal estimates and to the restriction to the dissipative case only and to  $\s<1$ and $\d\le 2$.

In fact, we can say better, and this is where we depart significantly from previous works,  by exploiting the fact that the expression  $\int_{\R^\d} |\nab H_{N, \veta}^\mu|^2
-\cd \sum_{i=1}^N   \g(\eta_i)$ is essentially {\it  decreasing} with respect to $\eta_i$ and constant for $\eta_i$ small enough that the $B(x_i, \eta_i)$'s are disjoint, see
Proposition~\ref{prop:monoto}. More precisely, we may set $\rr_i$ to be $1/4$ of the minimal distance from $x_i$ to all other points, and for $\eta_i\le \rr_i$, we have the {\it equality} (without limits)
\be \label{equality} F_N(X_N, \mu) = \frac{1}{\cd} \(
\int_{\R^\d} |\nab H_{N, \veta}^\mu|^2
-\cd \sum_{i=1}^N   \g(\eta_i)\)+ \text{explicit negligible terms}.\ee
In addition, an observation made in this paper for the first time is that when choosing precisely $\eta_i=\rr_i$,
the potentially large  terms $\int_{\R^\d} |\nab H_{N, \vec{\rr}}^\mu|^2$ and $\cd \sum_{i=1}^N   \g(\rr_i)$ are {\it separately} controlled by $C F_N$ (plus good terms) and conversely, see Corollary \ref{34}. 
It now suffices to control the left-hand side of \eqref{crucial} by $\int_{\R^\d} |\nab H_{N, \vec{\rr}}^\mu|^2$.
Let us emphasize that this choice of truncation  $\rr_i$ that depends on the point is (up to constants) the only one which is at the same time large enough so that  $\int_{\R^\d} |\nab H_{N, \veta}^\mu|^2$ is directly controlled by $F_N$ and small enough that the equality \eqref{equality} holds. In other words, since we do not have any bound from below on the distance between the points, a point-dependent truncation radius  is crucial. As a side note, the idea of using truncations for proving mean-field limits is common, however what is usually done is to truncate the interaction $\g$ itself (at lengthscales possibly depending on $N$), and try to take limits in the resulting flow. What we do here is very different: we do not modify the interaction but desingularize the charges themselves
as an intermediate step to control the singular terms.

In order to bound the left-hand side of \eqref{crucial},  
the key is then to interpret it as well as a single integral in  terms of $\nab H_N^\mu$
 in a suitable ``renormalized" way, more precisely in terms of the stress tensor associated to $\nab H_{N,\vec{\eta}}^\mu$, for $\eta_i\to 0$.  The quantity then obtained is this time not monotone in $\eta_i$, however, by carefully studying its variations in $\eta_i$ (this is the hardest part of the proof), we are able to control it by the (integral of the) stress tensor associated to $\nab H_{N, \vec{\rr}}$, which can in turn be bounded by $\int |\nab H_{N, \vecr}^\mu|^2$, and we conclude thanks to the fact 
that this is controlled by $CF_N$.

%  We do this via a truncation of $h^{\mu}$, at a lengthscale $\rr_i$ depending on the point $x_i$ and 
%equal to the minimal distance from $x_i$ to its nearest neighbors. An  observation made for the first time in this paper is  that when using these particular truncation parameters, the truncated energy can be controlled by the full energy and conversely (see Corollary \ref{34}). This would not be true when using truncation parameters that are independent of $i$, as we do not know any control on the minimal distances between points.
%The next crucial result is that, even though positivity is lost, when using the $\rr_i$'s as truncation parameters we can still control for each time slice the expression in \eqref{pt} by the modulated energy $F_N(X_N^t, \mu^t)$ itself, up to a small error (Proposition \ref{32}), and  provided the limiting solution is regular enough.
%This, which is the most difficult part of the proof, uses two ingredients: the first is to reexpress \eqref{pt}
%as a single integral in terms of a stress-energy tensor, and the second is to show that the expression in \eqref{pt} is in fact close to the same expression with truncated fields.

\subsection{The modulated free energy for the case with noise}\label{introsupp}
In this subsection, we explain the modulated free energy method of \cite{bjw}, which is posterior to the first version of this work and allows to treat the case of \eqref{mfgfnoise}.
In this case, the mean-field limit inherits an added Laplacian:
\be\label{mfn1}\partial_t \mu= \div ((\nab \g+\F)* \mu)\mu)+\theta \Delta \mu.\ee

 Consider $f_N(x_1, \dots, x_N)$ 
 a symmetric probability density on configurations in $(\R^\d)^N$, and let us again abbreviate $(x_1, \dots, x_N)$ by $X_N$. 
Let us introduce the relative entropy \footnote{Note that we take $N^2$ times the usual relative entropy, because all our quantities are $N^2$  times standard ones.} 
$$H_N(f_N |\mu^{\otimes N}) := N \int_{\R^{\d N}} f_N\log \frac{f_N}{\mu^{\otimes N}} dX_N$$
It is of course a way of measuring how 
close the distribution $f_N$ is to $\mu^{\otimes N}$.
Then consider $$K_N(f_N,\mu) = \int_{\R^{\d N}} f_N(X_N) F_N(X_N, \mu) dX_N$$
the expectation of our modulated energy $F_N$ with respect to $f_N$.
Bresch-Jabin-Wang introduce the modulated free energy 
\be \label{mfenergy}
\mathcal F_\theta(f_N, \mu) = \theta H_N(f_N |\mu^{\otimes N}) + K_N(f_N, \mu).\ee
It has exactly the structure of a free energy in statistical physics, i.e.~of the form energy plus temperature times entropy, and when temperature vanishes and $f_N$ concentrates on one configuration, it coincides with the regular modulated energy.

Consider now $f_N^t$ corresponding to the probability density of particles following the flow \eqref{mfgfnoise}, then by Ito calculus $f_N^t$ solves a Liouville or Kolmogorov equation 
\be\label{liou}
\partial_t f_N^t = \sum_{i=1}^N \div_i \(     \frac{1}{N} \sum_{i\neq j} \nab_i\mathcal H_N(x_i-x_j)     f_N^t (X_N) \)+ \theta 
\sum_{i=1}^N \Delta_{i} f_N^t .\ee 
Their crucial observation is that when combining the relative entropy and the modulated energy in exactly the way  of \eqref{mfenergy} and differentiating in time $\mathcal F_\theta(f_N^t, \mu^t)$ for $\mu^t $ a solution to the mean-field limit  \eqref{mfn1}  and $f_N^t$ a solution to the Liouville equation \eqref{liou}, the new and problematic terms  arising in $\partial_t H_N(f_N^t, \mu^t)$  from the presence of the noise (which were an obstacle to treat the Coulomb case in \cite{jabinwang2}) exactly cancel with the new terms arising in $\partial_t K_N(f_N^t, \mu^t)$ (this does not happen in the conservative case of \eqref{consnoise} though). This  allows them to obtain the following crucial identity:
\begin{multline}
\frac{d}{dt} \mathcal F_\theta (f_N^t, \mu^t) \\ \le -  \int \(\iint_{\triangle^c} 
\nab \g(x-y)\cdot  (\psi^t(x)-\psi^t(y)) d( \sum_{i=1}^N  \delta_{x_i}  - N\mu^t)(x) d(\sum_{i=1}^N \delta_{x_i} - N \mu^t)(y)\) df_N^t(X_N)
\end{multline}
with this time 
$$\psi^t=\nab h^{\mu^t}+ \theta \frac{\nab \mu^t}{\mu^t}.$$
Once this identity is observed, Proposition \ref{32} directly applies (if $\mu^t$ is assumed regular enough) and  yields
$$ \frac{d}{dt}\mathcal F_\theta(f_N^t, \mu^t) \le CK_N(f_N^t, \mu^t) +o(N^2) \le C \mathcal F_\theta (f_N^t, \mu^t)+ o(N^2)
 $$
which allows to directly conclude via Gronwall's lemma  the mean-field convergence in the case with noise.

\smallskip

The rest of the paper is organized as follows: we start by deriving the main result assuming the result of Proposition \ref{32}. In Section  \label{sec:def} we present the details of the electric formulation in the general Riesz case and prove the main properties of the modulated energy (monotonicity, bound from below, coerciveness).  We conclude in Section \ref{secp32} with the discussion of the stress-energy tensor and  the proof of Proposition \ref{32}. 
The paper finishes with Appendix \ref{app1} on the derivation of the Vlasov-Poisson system in the monokinetic case.
\medskip
%\smallskip

%\vskip .3cm

{\bf Acknowledgments:} I would like to thank Mitia Duerinckx for his careful reading and helpful suggestions and Pierre-Emmanuel Jabin for useful discussions on the work \cite{bjw,bjw2}.
This research was supported by the NSF grant DMS-1700278.
\section{Main proof}
In all the paper, we will use the notation $\indic_{\eqref{glog}}$ to indicate a term which is only present in the logarithmic cases \eqref{glog} and $\indic_{\s<\d-1}$ for a term which is present only if $\s<\d-1$. Below,  the principal value (which may be omitted for $\s <\d-1$) is defined by 
$P. V.\int_{\R^{\d} \backslash \{x\} } := \lim_{r\to 0} \int_{\R^{\d}\backslash B(x,r)} .$

Differentiating from formula \eqref{energydef}, we have 
\begin{lem}\label{lem31}
If $\XN^t$ is a solution of \eqref{mfgf}, then
\begin{multline}\label{f1}
\partial_t F_N(\XN^t, \mu^t)= - 2 N^2  \int_{\R^\d} \left|P.V. \int_{\R^\d\backslash \{x\} } \nab \g(x-y)d (\mu_N^t -  \mu^t)(y)\right|^2 d\mu_N^t(x) \\+ 2N^2 \int_{\R^\d} \F* (\mu^t-\mu_N^t) \cdot 
\nab (h^{\mu_N^t}-h^{\mu^t}) d\mu_N^t
\\- N^2 \iint_{\R^{\d}\times \R^{\d} \backslash \triangle} \( (\nab h^{\mu^t}+\F* \mu^t)(x)- (\nab h^{\mu^t} +\F*\mu^t) (y)\)\cdot  \nab \g(x-y) d (\mu_N^t - \mu^t)(x)d (\mu_N^t -  \mu^t)(y).
\end{multline} 
If $\XN^t$ is a solution of \eqref{cons}, 
then 
\begin{multline}\label{f2}
\partial_t F_N(\XN^t, \mu^t)=  -N^2 \iint_{\R^{\d}\times \R^{\d} \backslash \triangle} \mathbb{J}\( \nab h^{\mu^t}(x)- \nab h^{\mu^t} (y)\)\cdot  \nab \g(x-y) (\mu_N^t - \mu^t)(x) (\mu_N^t -  \mu^t)(y).
\end{multline} 
\end{lem}

\begin{remark}\label{remV}
In the case of an added term $\frac{1}{N}\sum_{i=1}^N\mathsf{V}(x_i)$ in the evolutions, one obtains an additional term 
$$ -2N^2 \iint_{\R^{\d}\times \R^{\d} \backslash \triangle} (\mathsf{V}(x)-\mathsf{V} (y))\cdot  \nab \g(x-y) d (\mu_N^t - \mu^t)(x)d (\mu_N^t -  \mu^t)(y)$$ which can be handled like the others using Proposition \ref{32} if $\mathsf{V}$ is globally Lipschitz.
\end{remark}
\begin{proof}
We note that if $\s\ge \d-1$, $\nab \g$ is not integrable near $0$, so $\nab \g* \mu$ should be understood in a distributional sense and $\mu\nab (\g* \mu)=\mu \g* \nab \mu$ as well, assuming that $\mu$ is regular enough.
We may also check that this distributional definition is equivalent to defining 
$$\nab h^\mu(x)= P.V. \int_{\R^{\d}\backslash \{x\}} \nab \g(x-y) d\mu(y).$$
In the case \eqref{mfgf}, we have
\begin{eqnarray*}
\partial_t F_N(\XN^t, \mu^t)& = & N^2 \partial_t \iint_{\R^{\d}\times \R^{\d}} \g(x-y) d\mu^t(x)d\mu^t(y)+ \partial_t \sum_{i\neq j} \g(x_i^t-x_j^t)
\\ &&- 2 N \partial_t\sum_{i=1}^N \int_{\R^{\d}} \g(x_i^t- y) d\mu^t(y)\\
& = & - 2N^2 \int_{\R^{\d}} |\nab h^{\mu^t} |^2  (x) d\mu^t(x)    - 2 N^2 \int_{\R^\d}  \nab h^{\mu^t} (x)\cdot \F* \mu^t(x) d\mu^t(x)  \\ & & -2\sum_{i=1}^N \left|\sum_{j\neq i} \nab \g(x_i^t-x_j^t) \right|^2 
  - 2\sum_{i=1}^N \(\sum_{j\neq i} \nab \g(x_i^t-x_j^t)   \cdot   \sum_{j\neq i} \F(x_i^t-x_j^t) \)
  \\&&+ 2N \sum_{i \neq j} \nab h^{\mu^t} (x_i^t) \cdot\( \nab \g(x_i^t-x_j^t) +\F(x_i^t-x_j^t)\)\\ && +2N \sum_{i=1}^N P.V. \int_{\R^{\d}\backslash \{x_i^t\}}  (\nab h^{\mu^t} +\F*\mu^t) (x) \cdot \nab \g(x-x_i^t) d\mu^t(x).\end{eqnarray*}
We then recombine the terms to obtain
\begin{eqnarray*}
\partial_t F_N(\XN^t, \mu^t) &= & -2N^2 \int_{\R^{\d}}\left|P.V. \int_{\R^{\d} \backslash \{x\}}  \nab \g(x-y) d(\mu_N^t-\mu^t) (y)\right|^2 d\mu_N^t (x) 
\\
 && -2N^2 \int_{\R^{\d}}| \nab h^{\mu^t}|^2  d\mu^t  -2N^2  \int_{\R^\d}  \nab h^{\mu^t} (x)\cdot \F* \mu^t(x) d\mu^t(x) 
 \\
&&-2N \int_{\R^\d} \( \nab h^{\mu_N^t}(x) \cdot  \int_{\R^\d \backslash \{x\}} \F(x-y) d\mu_N^t(y) \) d\mu_N^t(x)+2N^2 \int_{\R^{\d}}|  \nab h^{\mu^t}|^2  d\mu_N^t\\
& & + 2N^2 \int_{\R^{\d}}  \nab h^{\mu^t} (x) \cdot \int_{\R^{\d} \backslash \{x\} }(\nab \g+\F)(x-y) d\mu_N^t(y) d\mu_N^t(x)\\
&& +2 N^2 \int_{\R^{\d}} P.V. \int_{\R^{\d} \backslash \{y\}} (\nab h^{\mu^t}+\F*\mu^t)(x) \cdot  \nab\g(x-y) d\mu^t(x) d\mu_N^t(y).\end{eqnarray*}
We next recognize that all the terms except the first can be recombined and symmetrized  into 
\begin{multline*}
- N^2  \iint_{\triangle^c} \( (\nab h^{\mu^t} +\F* \mu^t) (x)-( \nab h^{\mu^t}+\F* \mu^t )(y) \)  \cdot \nab \g(x-y) d(\mu_N^t-  \mu^t)(x) d(\mu_N^t-  \mu^t)(y)\\
+ 2N^2 \int_{\R^\d} \F * (\mu^t-\mu_N^t) \cdot \nab ( h^{\mu_N^t}-h^{\mu^t} ) d\mu_N^t\end{multline*}
which gives the desired formula.

In the case \eqref{cons} we have
\begin{multline*}
\partial_t F_N(\XN^t, \mu^t)= N^2 \partial_t \iint \g(x-y) d\mu^t(x)d\mu^t(y)+ \partial_t \sum_{i\neq j} \g(x_i^t-x_j^t)
- 2 N \partial_t\sum_{i=1}^N \int_{\R^{\d}} \g(x_i^t- y) d\mu^t(y)\\
= 2N \sum_{i\neq j} \nab h^{\mu^t} (x_i^t)\cdot \mathbb{J}\nab \g(x_i^t-x_j^t)+ 2N\sum_{i=1}^N  P.V. \int_{\R^{\d}\backslash\{x_i^t\}}  \mathbb{ J} \nab h^{\mu^t} (x)\cdot \nab \g(x-x_i^t) d\mu^t(x)\end{multline*}
We then rewrite this as 
\begin{eqnarray*}
\partial_t F_N(\XN^t, \mu^t)& = & 2N^2 \int_{\R^{\d}}  \nab h^{\mu^t}(x) \cdot \int_{\R^{\d} \backslash \{x\}}  \mathbb{J}\nab \g(x-y) d\mu_N^t(y) d\mu_N^t(x)
\\
&& +2 N^2 \int_{\R^{\d}}  P.V. \int_{\R^{\d}\backslash\{y\}}  \mathbb{ J} \nab h^{\mu^t} (x)\cdot \nab \g(x-y) d\mu^t(x) d\mu_N^t(y)
.\end{eqnarray*}
By antisymmetry of $\mathbb{J}$, we recognize that the right-hand side can be symmetrized into
$$-N^2 \iint_{\triangle^c}  \mathbb{J}\( \nab h^{\mu^t}(x)- \nab h^{\mu^t}(y)\)  \cdot \nab \g(x-y) d(\mu_N^t-  \mu^t)(x) d(\mu_N^t-  \mu^t)(y).$$
\end{proof}

The main point is thus to control the last term in the right-hand side of \eqref{f1} or \eqref{f2} which is done via  Proposition \ref{32}.

For the dissipative case with the added force, we will also need
\begin{lem}\label{lemforce}
Assume $ \F\in \dot{H}^{\frac{\d-\s}{2}} (\R^\d)\cap C^{0,\alpha}(\R^\d)$ for some $\alpha>0$. 
Then there exists $\lambda>0$  and $C>0$ depending only on $\alpha, \s, \d$ such that  for every $t$,
\begin{multline*}
N^2 \int_{\R^\d} \F* (\mu^t-\mu_N^t)\cdot \nab (h^{\mu_N^t}-h^{\mu^t}) d\mu_N^t
 \le  N^2 \int_{\R^\d} \left|P.V. \int_{\R^\d\backslash \{x\} } \nab \g(x-y)d (\mu_N^t -  \mu^t)(y)\right|^2 d\mu_N^t(x)  
\\+
C \|\F\|^2_{\dot{H}^{\frac{\d-\s}{2}} (\R^\d)}
\( F_N(\XN^t, \mu^t)+ (1+\|\mu^t\|_{L^\infty} )  N^{1+\frac{\s}{\d}}+2 \( \frac{N}{\d}\log N\) \indic_{\eqref{glog}} \)
+  C\|\F\|^2_{C^{0, \alpha}(\R^\d)} N^{2-\frac{2\lambda}{\d}}.\end{multline*}
\end{lem}

Noting that by assumption $\mu^t \in \cap_{p=1}^\infty L^p(\R^\d)$ and taking $p$ to be the conjuguate exponent to $q$, 
$$\|\nab \F* \mu^t\|_{L^\infty} \le \|\nab \F\|_{L^q} \|\mu^t\|_{L^p} ,$$ we then immediately deduce from Lemma \ref{lem31}, Lemma \ref{lemforce} and Proposition \ref{32} that 
\begin{multline*}\partial_t F_N(\XN^t, \mu^t)\le C \(\|\nab^2 h^{\mu^t}\|_{L^\infty(\R^\d)}+\|\nab \F\|_{L^q(\R^\d)}+ \|\F\|^2_{\dot{H}^{\frac{\d-\s}{2}} (\R^\d)}\)\\ 
\times \Big[ \( F_N(\XN^t, \mu^t) + (1+\|\mu^t\|_{L^\infty} )  N^{2-\frac{\d-\s}{\d(\d+1)}}+ N^{\frac32}+2\(\frac{N}{\d}\log N\) \indic_{\eqref{glog}} \) + 
 C  \|\F\|^2_{C^{0, \alpha}(\R^\d)} N^{2-\frac{2\lambda}{\d}}  \\
+ C  \begin{cases}
  (1+ \|\mu^t\|_{C^{\sigma}}) N^{2-\frac{1}{\d}}+
 \(  \|\nab h^{\mu^t} \|_{L^\infty} +  \|\nab^2 h^{\mu^t} \|_{L^\infty}\) \|\mu\|_{C^{\sigma}} N^{1+\frac{\s +1-\sigma}{\d}}\Big] \  \text{if} \  \s \ge \d-1\\
 (1+ \|\mu^t\|_{L^\infty}) N^{2-\frac{1}{\d}}
+
   \|\nab h^{\mu^t} \|_{L^\infty} \|\mu^t\|_{L^\infty}  N^{1+\frac{\s+1}{\d}} +  \|\nab^2 h^{\mu^t}\|_{L^\infty}\|\mu^t\|_{L^\infty} ) N^{1+\frac{\s}{\d}} \Big]\ \text{if} \  \s < \d-1.
\end{cases}  \end{multline*}
Since $\s<\d$ and $\sigma>\s-\d+1$,  this implies by Gronwall's lemma and in view of \eqref{nmut} that for every $t\le T$,
\begin{equation*}F_N(\XN^t, \mu^t) \le \(F_N(\XN^0, \mu^0) +C_1 N^{\beta}\)  e^{C_2 t} \quad \text{for some } \beta<2.\end{equation*}
%where 
%\begin{align*}
%& M_1=C\|\nab^2 h^{\mu^t}\|_{L^\infty} \qquad M_2=  C\(  \|\nab h^{\mu^t}\|_{L^\infty}\|\mu\|_{C^{\theta}}+ \|\nab^2h^{\mu^t} \|_{L^\infty} \|\mu\|_{L^\infty} \)  \\
%& M_3= C\|\nab^2 h^{\mu^t}\|_{L^\infty} (1+\|\mu^t \|_{C^{\theta}})  \qquad M_4= C \|\mu\|_{C^\theta} \|\nab h^{\mu^t}\|_{L^\infty}. \end{align*}
In view 
%of  \eqref{hmureg2} and 
Proposition \ref{procoer} below, this proves the main theorem.

\section{Formulation via the electric potential} \label{sec:def}

 \subsection{The extension representation for the fractional Laplacian}
\label{sec:deflap}
%In the next two sections, we recall elements from \cite{petrache2014next}.
%Our method of proof relies on expressing the interaction part of the Hamiltonian  as a quadratic integral  of the potential  generated by the point configuration via
%$$ \g \star \sum_{i} \delta_{x_i}$$
%(where $\star$ denotes the convolution product) and expanding this integral interaction to next order in $N$.
In general, the kernel $\g$ is not the convolution kernel of a local operator, but rather of a fractional Laplacian. Here we use the \textit{extension representation} popularized by \cite{caffsilvestre}: by adding one space variable  $y\in \R$ to the space $\R^{\d}$, the nonlocal operator can be transformed into a local operator of the form $-\div (\yg \nab \cdot)$.

In what follows, $\k$ will denote the dimension extension. We will take $\k=0 $ in the Coulomb cases  for which $\g$ itself is the kernel of a local operator. In all other cases, we will  take $\k=1$.

 For now, points in the space $\R^{\d}$ will be denoted by $x$, and points in the extended space $\R^{\d+\k}$ by $X$, with $X=(x,z)$, $x\in \R^{\d}$, $z\in \R^\k$. We will often identify $\R^{\d} \times \{0\}$ and $\R^{\d}$ and thus $(x_i,0)$ with $x_i$.

% We will also use
%the notation 
%\begin{equation}\label{notkr}
%K_R=[-R/2,R/2]^d  \mbox{ as well as} \ K_R=[-R/2,R/2]^d \times \{0\}\end{equation}
%\begin{equation}\label{tkr}
%\tilde{K}_R= [-R/2,R/2]^d \times (-R/2,R/2)\end{equation}
If $\gamma$ is chosen such that 
\begin{equation}\label{gs}
\d-2+\k+ \gamma  =\s,
\end{equation}
then, given a probability measure $\mu$ on $\R^{\d}$, the $\g$-potential generated by $\mu$, defined in $\R^{\d}$ by
\begin{equation*}
{h}^{\mu}(x) := \int_{\R^{\d}} \g(x-\tilde{x}) \, d\mu(\tilde{x})
\end{equation*}
can be extended to a function $\mathsf{h}^\mu$ on $\R^{\d+\k}$ defined by 
\begin{equation*}
\mathsf{h}^{\mu}(X) := \int_{\R^{\d}} \g(X - (\tilde{x},0)) \, d\mu(\tilde{x}), 
\end{equation*}
and this function satisfies 
\begin{equation}
\label{divh}
- \div (\yg \nab \mathsf{h}^{\mu})=  \c {\mu} \delta_{\R^{\d}}
\end{equation}
where by $\delta_{\R^{\d}}$ we mean the uniform  measure on $\R^{\d}\times \{0\}$. The corresponding values of the constants $\c$ are given in \cite[Section 1.2]{ps}.  In particular, the potential $\g$ seen as a function of $\R^{\d+\k}$ satisfies
 \begin{equation}\label{eqg}
 - \div (\yg \nab \g)= \c \delta_0.
 \end{equation}

%In other words  $\mu \drd$ 
%acts on test functions $\vp$ by 
%$$\int_{\R^{\d+\k}} \vp (X)d (\mu \drd)(X) = \int_{\R^{\d}} \vp(x, 0) \, d\mu(x),$$
% and 
% \begin{equation}
% \cds= \left\lbrace\begin{array}{ll}2s\,\frac{2\pi^{\frac{d}{2}}\Gamma\left(\frac{s+2-d}{2}\right)}{\Gamma\left(\frac{s+2}{2}\right)}&\text{ for }s>\max(0,d-2)\ ,\\[3mm]
%                      (d-2)\frac{2\pi^{\frac{d}{2}}}{\Gamma(d/2)}&\text{ for }s=d-2>0\ ,\\[3mm]
%                      2\pi&\text{ in cases \eqref{wlog}, \eqref{wlog2d}}\ .
%                     \end{array}\g%\right.
% \end{equation}
% In order to recover the Coulomb cases, it suffices to take $k=\gamma=0$, in which case we retrieve the fact that $g$ is a multiple of the fundamental solution of the Laplacian. If $s>d-2$ we take $k=1$ and 
%$\gamma$ satisfying \eqref{gs}. In the case \eqref{wlog}, we note that $\g(x)=- \log |x|$ appears as the $y=0$ restriction 
%of $-\log |X|$, which is (up to a factor $2\pi$) the fundamental solution  to the Laplacian operator in dimension $d+k=2$. In this case, we may thus choose $k=1$ and $\gamma=0$, $\cds=c_{1,0}=2\pi$, and the potential $H^\mu=\g \star \mu$ still satisfies \eqref{divh}, while $\g$ still satisfies \eqref{eqg}. 

To summarize, we will take 
\begin{itemize}
\item $\k=0, \gamma=0$ in the Coulomb cases. 
{\it The reader only interested in the Coulomb cases may thus just ignore the $\k$ and the weight $\yg$ in all the integrals.}
\item $\k=1, \gamma = \s - \d + 2 - \k$ in the  Riesz cases and in the one-dimensional logarithmic case (then we mean $\s=0$). Note that our  assumption $(\d-2)_+\le \s<\d$  implies  that $\gamma$ is always in $(-1,1)$. We refer to \cite[Section 1.2]{ps} for more details.
\end{itemize}

We now make a remark on the regularity of $h^\mu$:
\begin{lem}\label{lemregu}
Assume $\mu$ is a probability density in $ C^{\theta}(\R^{\d})$ for some $\theta>\s - \d+2 $, then we have
\be\label{hmureg}\|\nab {h}^\mu\|_{L^\infty (\R^{\d})}\le  C\( \|\mu\|_{C^{\theta-1}(\R^{\d})}+\|\mu\|_{L^1(\R^{\d})}\) , \ee
and 
\be \label{hmureg2} \|\nab^2 h^\mu\|_{L^\infty(\R^{\d})} \le C\(    \|\mu\|_{C^{\theta}(\R^{\d})}+\|\mu\|_{L^1(\R^{\d})}\)  . \ee
\end{lem}
\begin{proof}
As is well known, $\g$ is (up to a constant) the kernel of $\Delta^{\frac{\d-\s}{2}}$, hence
 ${h}^\mu= \c \Delta^{\frac{\s-\d}{2}} \mu$ and the relations follow (cf. also \cite[Lemma 2.5]{du}).
%We will also let
%\be \label{defa} \alpha=\frac{\d-\s}{2}\ee
\end{proof}

\subsection{Electring rewriting of the energy} \label{sec:nextorder}
We briefly recall the procedure used in \cite{rs,ps} for truncating the interaction or, equivalently, spreading out the point charges.
It will also be crucial to use the variant introduced in \cite{lsz,ls2} where we let the truncation distance depend on the point. 

For any $\eta \in (0,1)$, we define
\begin{equation}
\label{def:truncation} \g_{\eta} := \min(\g, \g(\eta)), \quad \f_{\eta} := \g - \g_{\eta}
\end{equation}
and 
\begin{equation} \label{defde}
\delta_0^{(\eta)}:= - \frac{1}{\c} \div (\yg \nab \g_{\eta}), 
\end{equation}
which is a positive  measure supported on $\partial  B(0, \eta)$.
\begin{remark}\label{nonsmooth}
This nonsmooth truncation of $\g_\eta$ can be replaced with no change by  a smooth one such that 
$$\g_\eta(x)=\g(x) \ \text{for} \  |x|\ge \eta, \quad \g_\eta(x)=cst \ \text{for } |x|\le \eta-\ep\qquad \ep<\hal \eta$$
and this way 
$\delta_0^{(\eta)}$ gets replaced by a probability measure with a regular density supported in $B(0, \eta)\backslash B(0, \eta-\ep)$.
We make this modification whenever the integrals against the singular measures may not be well-defined.\end{remark}

We will also let 
\be \label{fae}
\f_{\alpha, \eta}:=\f_\alpha-\f_\eta=\g_\eta-\g_\alpha,\ee
and we observe that $\f_{\alpha,\eta}$ has the sign of $\alpha-\eta$, vanishes outside $B(0,\max( \alpha, \eta)) $, and satisfies
\be\label{pfa} \g*(\delta_x^{(\eta)}-\delta_x^{(\alpha)})= \f_{\alpha, \eta}(\cdot - x)\ee
and \be\label{sfe}
-\div(\yg \nab \f_{\alpha, \eta})= \c (\delta_0^{(\eta)}-\delta_0^{(\alpha)}).\ee

%We start with the energy splitting formula of \cite{ps} that separates fixed leading order terms from variable next order ones.
For any configuration $\XN=(x_1, \dots, x_N)$, we define  for any $i$ the minimal distance
\be \label{defri} \rr_i= \min\( \frac{1}{4}\min_{j \neq i} |x_i-x_j| , N^{-\frac1\d}\).\ee
For any $\vec{\eta}=(\eta_1, \dots, \eta_N) \in \R^N$ and measure $\mu$, we define the electric potential
\be \label{defHN0} H_{N}^{\mu}[\XN]= \int_{\R^{\d+\k}} \g(x-y) d \(\sum_{i=1}^N \delta_{x_i}- N \mu\drd\) (y)\ee
and the  truncated  potential 
\be\label{defHN} H_{N,\vec{\eta}}^{\mu}[\XN]= \int_{\R^{\d+\k}} \g(x-y)d \(\sum_{i=1}^N \delta_{x_i}^{(\eta_i)} - N \mu\drd \) (y),\ee
 where we will  quickly drop the dependence in $\XN$. We note that 
 \be\label{HHf}
  H_{N, \vec{\eta}}^\mu[\XN]= H_{N}^\mu[\XN]- \sum_{i=1}^N \f_\eta(x-x_i).\ee
  These functions are viewed in the extended space $\R^{\d+\k}$ as described in the previous subsection, and solve  \be \label{eqnhne0} 
 -\div (\yg \nab H_{N}^{\mu})=\c \(\sum_{i=1}^N \delta_{x_i} - N \mu\drd\)\quad \text{in } \R^{\d+\k},\ee
 and
 \be \label{eqhne}
 -\div (\yg \nab H_{N,\vec{\eta}}^{\mu})=\c \(\sum_{i=1}^N \delta_{x_i}^{(\eta_i)} - N \mu\drd\)\quad \text{in } \R^{\d+\k}.\ee

 The following proposition shows how to express $F_N$ in terms of the truncated electric fields $\nab H_{N,\vec{\eta}}^\mu$. In addition, 
 we show that the quantities  $$ \int_{\R^{\d+\k}}\yg |\nab H_{N, \vec{\eta}} ^{\mu}|^2  - \c\sum_{i=1}^N \g(\eta_i)  $$ converge almost  monotonically (i.e. up to a small error) to $F_N$, while the discrepancy between the two can serve to control the energy of close pairs of points. 
 \begin{prop} \label{prop:monoto}
Let $\mu$ be a bounded probability density on $\R^{\d}$ and $\XN$ be in $(\R^{\d})^N$. We may re-write $\FN(\XN, \mu)$ as  
\begin{equation}
\label{def:FNbis}
\FN(\XN, \mu) := \frac{1}{\c} \lim_{\eta \to 0} \left(\int_{\R^{\d+\k}}\yg |\nab H_{N, \vec{\eta}} ^{\mu}|^2  - \c\sum_{i=1}^N \g(\eta_i)  \right),
\end{equation}and for any $\veta$ 
we have the bound
 \begin{multline}
\label{fnmeta2}
\sum_{i\neq j}  \(\g(x_i-x_j)- \g(\eta_i)\)_+\\ \le   \FN(\XN,\mu)   -\(\frac{1}{\c} \int_{\R^{\d+\k}}\yg |\nab H_{N,\vec{\eta}}^\mu|^2 -\sum_{i=1}^N \g(\eta_i)\) +C N\|\mu\|_{L^\infty} \sum_{i=1}^N \eta_i^{\d-\s} ,
\end{multline}
for some  $C$ depending only on $\d$ and $\s$. \end{prop}
The proof, which is an adaptation and improvement of \cite{ps,ls2},
 is postponed to Section~\ref{app}. 
 
  What makes our main proof work is the ability to find  some choice of truncation  $\vec{\eta}$ such that 
  $\int_{\R^{\d+\k}} \yg |\nab H_{N, \vec{\eta}} ^{\mu}|^2$ ({\it without} the renormalizing term $-\c \sum_{i=1}^N \g(\eta_i) $) is controlled by $F_N(\XN, \mu)$ {\it and} the balls $B(x_i, \eta_i)$ are disjoint. In view of  \eqref{fnmeta2} the former could easily be achieved by taking the $\eta_i$'s large enough, say $\eta_i = N^{-1/\d}$, but the  balls would not necessarily be disjoint. Instead the choice of $\eta_i = \rr_i$ where $\rr_i$ are the minimal distances as in \eqref{defri} allows to fulfill both requirements, as seen in  the following   

\begin{coro}\label{34}
Under the same assumptions, we have  
\be\label{bgr}
\sum_{i=1}^N \g(\rr_i) \le C \( F_N(\XN, \mu)+ (1+\|\mu\|_{L^\infty} ) N^{1+\frac{\s}{\d}} +\( \frac{N}{\d}\log N\) \indic_{\eqref{glog}} \)+  C\( \frac{N}{\d}\log N\) \indic_{\eqref{glog}} 
\ee
and
\be \label{bornehnr}
\int_{\R^{\d+\k}}\yg |\nab H_{N ,\vec{\rr}}^\mu|^2\le C\( F_N(\XN, \mu)+ (1+\|\mu\|_{L^\infty} )  N^{1+\frac{\s}{\d}}+ \( \frac{N}{\d}\log N\) \indic_{\eqref{glog}} \)\ee
for some $C$ depending only on $\s, \d$.
\end{coro}
\begin{proof}
Let us choose  $\eta_i =N^{-1/\d}$ for all $i$  in \eqref{fnmeta2} and observe that for each $i$, by definition \eqref{defri} there exists $j\neq i$ such that 
$(\g(|x_i-x_j|)- \g(N^{-1/\d}))_+ =    (\g(4\rr_i)   - \g(N^{-1/\d}) )_+ $. We may thus write  that 
\be\label{lb1}\sum_{i=1}^N ( \g(4 \rr_i)-\g(N^{-\frac{1}{\d}}) )_+\le  F_N( \XN, \mu) - \frac{1}{\c}\int_{\R^{\d+\k}}\yg |\nab H_{N ,\vec{\eta}}|^2 +   N \g(N^{-\frac{1}{\d}}  ) + O(N \|\mu\|_{L^\infty} ) N^{\frac{\s}{\d}} . \ee
from which \eqref{bgr} follows.
% It also follows that 
%\begin{equation}\label{lbF}
%F_N(\XN,\mu) \ge - N \g(N^{-\frac{1}{\d}})-CN^{\frac{\s}{\d}} \|\mu\|_{L^\infty} \ge - C N^{1+\frac{\s}{\d}}- \frac{N}{\d} \log N \indic_{\eqref{glog}}.\end{equation}

Let us next choose   $\eta_i=\rr_i$ in \eqref{fnmeta2}. Using that $\rr_i \le N^{-1/\d}$, this yields
$$ 0 \le F_N(\XN, \mu) -  \frac{1}{\c} \int_{\R^{\d+\k}}\yg |\nab H_{N,\vec{\rr}}|^2 + \sum_{i=1}^N \g(\rr_i) + O(N\|\mu\|_{L^\infty} ) N^{\frac{\s}{\d}} .$$
Combining with \eqref{bgr}, \eqref{bornehnr} follows.

\end{proof}
 
 From \eqref{bornehnr} we directly obtain that $F_N$ is bounded below:
 \begin{coro}\label{corominob}
 Under the same assumptions we have 
 \be  F_N(X_N, \mu) \ge -\(\frac{N}{\d}\log N\)\indic_{\eqref{glog}} - C N^{1+\frac\s\d} 
 \ee for some $C>0$ depending only on $\d, \s$ and $\|\mu\|_{L^\infty}.$
 \end{coro}

\subsection{Coerciveness of the modulated energy}
Here we prove that  the modulated energy does metrize the convergence of $\mu_N^t$ to $\mu^t$. \begin{prop}
\label{procoer}
For any $0<\alpha\le 1$, there exists $\lambda>0$ and $C>0$ depending only on $\alpha$, $\d$, $\s$, such that 
for any $X_N\in (\R^\d)^N$, any probability density $\mu$, and any   $\xi \in C^\infty (\R^\d)$, we have
\begin{multline}\label{rhspro}
\left|\int_{\R^\d} \xi d \( \sum_{i=1}^N\delta_{x_i}-N\mu\)\right|\le 
 C \|\xi\|_{C^{0, \alpha}(\R^\d)} N^{1-\frac{\lambda}{\d}}\\+
C \|\xi\|_{\dot{H}^{\frac{\d-\s}{2}} (\R^\d)}
\( F_N(\XN, \mu)+ (1+\|\mu\|_{L^\infty} )  N^{1+\frac{\s}{\d}}+2 \( \frac{N}{\d}\log N\) \indic_{\eqref{glog}} \)^{\hal}.\end{multline}
In particular, if $\frac{1}{N^2} F_N(X_N, \mu) \to 0$ as $N \to \infty$, we have that 
\be\frac{1}{N} \sum_{i=1}^N \delta_{x_i} \rightharpoonup  \mu\quad \text{in the weak sense}.\ee
\end{prop}
\begin{proof}
Let $\xi$ be a smooth test function on $\R^\d$. Let $\bar\xi$ denote an extension of $\xi$ to $\R^{\d+\k}$ satisfying 
$$-\div(\yg \nab \bar \xi)=0\ \text{in} \ \{z\neq 0\}.$$
By \cite{kfs}, $\yg$ being a Muckenhoupt $A_2$ weight, the function $\bar \xi$ is in $C^{0, \lambda}(\R^{\d+\k})$ for some $\lambda>0$ depending on the other parameters, with $\|\bar \xi\|_{C^{0, \lambda}} \le C\|\xi\|_{C^{0, \alpha}}$. This can also be seen from the Poisson kernel representation given in \cite{caffsilvestre}.
In addition, we also have (this can be seen in Fourier, see \cite[Section 3.2]{caffsilvestre} 
\be \label{hs}\int_{\R^{\d+\k}}\yg |\nab \bar \xi|^2 = C \|\xi\|^2_{\dot{H}^{\frac{\d-\s}{2}} (\R^\d)}.\ee

Using \eqref{eqhne} let us write  for any probability density $\mu$
\begin{equation}\label{ghd}
\int_{\R^\d} \xi\,d\(  \sum_{i=1}^N\delta_{x_i}-N\mu\)  =\int_{\R^\d} \xi d\(  \sum_{i=1}^N \delta_{x_i}-\delta_{x_i}^{(\rr_i)}\) 
   -  \frac{1}{\c} \int_{\R^{\d+\k} }\bar \xi \, \div (\yg \nab H_{N,\vecr}^\mu[X_N])
 .\end{equation}
 For the first term in the right-hand side we use the H\"older continuity of $\bar \xi$ and the fact that $\delta_{x_i}^{(\rr_i)}$ is supported in $B(x_i,\rr_i)$ and write 
 \be\label{1es}\left|\int_{\R^\d} \xi  d \sum_{i=1}^N \( \delta_{x_i}-\delta_{x_i}^{(\rr_i)}\) \right|\le C \|\xi\|_{C^{0, \alpha}(\R^\d)}\sum_{i=1}^N \rr_i^\lambda\le 
 C \|\xi\|_{C^{0, \alpha}(\R^\d)} N^{1-\frac{\lambda}{\d}}.\ee 
  For the second term, we integrate by parts and use the Cauchy-Schwarz inequality to write 
 \begin{multline}
\left|  \int_{\R^{\d+\k} }\bar \xi \, \div (\yg \nab H_{N,\vecr}^\mu[X_N])\right|= \left| \int_{\R^{\d+\k}}\yg \nab \bar \xi\cdot \nab H_{N,\vecr}^{\mu}\right|
\\
\le\( \int_{\R^{\d+\k}} \yg |\nab \bar \xi|^2 \)^\hal \( \int_{\R^{\d+\k}}\yg |\nab H_{N, \vecr}^{\mu} |^2 \)^\hal\end{multline}
 In view of  \eqref{bornehnr} and \eqref{hs} we thus find
 \begin{multline}\label{es2}
\left|  \int_{\R^{\d+\k} }\bar \xi\,  \div (\yg \nab H_{N,\vecr}^\mu[X_N])\right|
\\ \le C \|\xi\|_{\dot{H}^{\frac{\d-\s}{2}} (\R^\d)}\( F_N(X_N, \mu) + (1+\|\mu\|_{L^\infty}) N^{1+\frac\s\d} +2 \(\frac{N}{\d} \log N\) \indic_{\eqref{glog}}\)^{\hal}.\end{multline}
Inserting \eqref{1es} and \eqref{es2} into \eqref{ghd} we conclude the result.
\end{proof}

\vskip -1cm
\begin{remark}\label{remchaos}In a density formulation aiming at proving propagation of chaos,  arguing exactly as in  \cite[Lemma 8.4]{rs} for instance,  we may deduce from this result and the main theorem the convergence  of the $k$-marginal densities %(or $k$-particles reduced density) functions
in the dual of some Sobolev space, with rate $k/N$ times the right-hand side of \eqref{rhspro}.
\end{remark}

\subsection{Proof of Lemma \ref{lemforce}}
Using the Cauchy-Schwarz inequality we have the bound
\begin{multline*}
\iint_{\triangle^c} \F* (\mu^t-\mu_N^t) \cdot \nab \g(x-y) d(\mu_N^t-\mu^t) (y) d\mu_N^t(x)\\
\le \( \int_{\R^\d} |\F*(\mu^t-\mu_N^t) |^2 d\mu_N^t\)^\hal \( \int\left| P.V. \int_{ \R^\d\backslash \{x\}}\nab \g(x-y) d(\mu_N^t-\mu^t)(y)\right|^2 d\mu_N^t(x)\)^\hal\end{multline*}
thus to prove the lemma, it suffices to show that 
\begin{multline} \label{presu}
N\|\F*(\mu^t-\mu_N^t)\|_{L^\infty} \le  C \|\F\|_{C^{0, \alpha}(\R^\d)} N^{1-\frac{\lambda}{\d}}\\+
C \|\F\|_{\dot{H}^{\frac{\d-\s}{2}} (\R^\d)}
\( F_N(\XN, \mu)+ (1+\|\mu\|_{L^\infty} )  N^{1+\frac{\s}{\d}}+2 \( \frac{N}{\d}\log N\) \indic_{\eqref{glog}} \)^{\hal}
,\end{multline} which is a direct consequence of \eqref{rhspro}.
\section{Proof of Proposition \ref{32}}
\label{secp32}

\subsection{Stress-energy tensor}
\begin{defi}
For any functions $h, f$ in $\R^{\d+\k}$ such that  $\int_{\R^{\d+\k}} \yg|\nab h|^2$ and $\int_{\R^{\d+\k}} \yg |\nab f|^2$ are finite, we 
define the stress tensor  $[h,f]$ as the $(\d+\k )\times (\d+\k)$ tensor
\be  \label{stresst}
[h,f]=\yg \( \partial_i  h \partial_j  f + \partial_i h \partial_j f \) -  \yg \nab h \cdot \nab f \delta_{ij}
 \ee where $\delta_{ij}=1$ if $i=j$ and $0$ otherwise.
 \end{defi} 
 We note that
 \begin{lem}  If $h  $ and $f$  are regular enough,  we have
 \be \label{divt}
\div [h, f]=\div(\yg \nab h) \nab f+\div(\yg \nab f) \nab h    -  \nab \yg \nab h \cdot \nab f
\ee
where $\div T$ here denotes the vector with components $\sum_i \partial_i T_{ij}$, with $j$ ranging from $1 $ to $\d + \k$.\end{lem}
\begin{proof}
This is a direct computation. Below, all sums range from $1$ to $\d+\k$.
\begin{multline*}
\sum_i \partial_i [h, f]_{ij}\\ =
\sum_i \left[\partial_i (\yg \partial_i h) \partial_j f +  \partial_i (\yg \partial_i f) \partial_j h
 +\yg  \partial_{ij} h \partial_i f+ \yg \partial_{ij} f   \partial_i h\right]   - \partial_j \(\yg \sum_i \partial_i h \partial_i f\)\\
=  \div (\yg \nab h)+ \div (\yg \nab f) -  \nab h \cdot \nab f \partial_j \yg .\end{multline*}\end{proof}

In view of \eqref{divt}, we have
\begin{lem} \label{lem44}
Let $\psi : \R^{\d}\to\R^{\d}$ be Lipschitz, and if $\k=1$ let $\hat \psi$ be an extension of it to a map from $\R^{\d+\k}$ to $\R^{\d+\k}$, whose last component identically vanishes, which tends to $0$ as $|z|\to \infty$ and has the same pointwise and Lipschitz bounds as $\psi$. \footnote{Such an extension exists, for instance by solving the $\infty$-Laplacian in a strip, which provides an ``absolutely minimal Lipschitz extension"}
For  any measures $\mu,\nu $ on $\R^{\d+\k}$, if 
$-\div (\yg \nab \mathsf{h}^\mu)= \c \mu$ and $-\div (\yg \nab \mathsf{h}^\nu)=\c \nu$, and  assuming that   $\int_{\R^{\d+\k}} \yg |\nab \mathsf{h}^\mu|^2 $ and $ \int_{\R^{\d+\k}} \yg |\nab \mathsf{h}^\nu|^2$  are finite and the left-hand side in \eqref{identt} is well-defined, we have 
\be\label{identt}
\iint_{\R^{\d+\k}\times \R^{\d+\k}} (\hat\psi(x) -\hat \psi(y) )\cdot \nab \g(x-y)  d\mu(x) d\nu(y) = \frac{1}{\c} \int_{\R^{\d+\k}} \nab \hat  \psi(x) : [\mathsf{h}^\mu, \mathsf{h}^\nu].\ee
\end{lem}
\begin{proof}
If $\mu$ is smooth enough then we may use   \eqref{divt}  to write 
\begin{eqnarray*}\iint_{\R^{\d+\k}\times \R^{\d+\k} } (\hat \psi(x) -\hat \psi(y) )\cdot \nab \g(x-y)  d\mu(x) d\nu(y) & =&
  \int_{\R^{\d+\k}} \hat \psi \cdot (\nab \mathsf{h}^{\mu} d \nu+ \nab\mathsf{h}^\nu d\mu)\\
&=& -\frac{1}{\c} \int_{\R^{\d+\k}} \hat\psi \cdot \div [\mathsf{h}^\mu, \mathsf{h}^\nu]\end{eqnarray*}
since  the last component of  $\hat \psi$ vanishes identically.
Integrating by parts, we obtain 
$$\iint_{\R^{\d+\k}\times \R^{\d+\k} }( \hat\psi(x)-\hat \psi(y)) \cdot \nab \g(x-y)  d\mu(x) d\nu(y) = \frac{1}{\c} \int_{\R^{\d+\k}} \nab \hat \psi :  [\mathsf{h}^\mu, \mathsf{h}^\nu].$$
By density, we may extend this relation to all measures $\mu, \nu$ such that both sides of \eqref{identt} make sense.
\end{proof}

\subsection{Proof of Proposition \ref{32}}
We now proceed to the proof. Given the Lipschitz map $\psi:\R^{\d}\to \R^{\d}$, we choose an extension $\hat \psi$ to $\R^{\d+\k}$ which satisfies the same conditions as in Lemma \ref{lem44}.

{\bf Step 1}: {\it renormalizing  the quantity and expressing it with the stress-energy tensor.}
 Clearly, 
\begin{multline}\label{45}
\iint_{\triangle^c}
 \( \psi(x)- \psi(y)\)\cdot  \nab \g(x-y) d(\sum_{i=1}^N \delta_{x_i}  - N\mu)(x) d(\sum_{i=1}^N \delta_{x_i} - N \mu)(y) \\= 
 \lim_{\eta\to 0}\Big[ \iint_{\R^{\d+\k}\times \R^{\d+\k}} 
 \( \hat \psi(x)- \hat \psi(y)\)\cdot  \nab \g(x-y) d(\sum_{i=1}^N \delta_{x_i}^{(\eta)} - N\mu\drd)(x) d(\sum_{i=1}^N \delta_{x_i}^{(\eta)} - N \mu\drd )(y)
\\ - \sum_{i=1}^N  \iint_{\R^{\d+\k}\times \R^{\d+\k}} 
 \( \hat\psi(x)- \hat \psi(y)\)\cdot  \nab \g(x-y) d \delta_{x_i}^{(\eta)} (x)d \delta_{x_i}^{(\eta)} (y)\Big].
\end{multline} 
Applying Lemma \ref{lem44}, in view of \eqref{defHN}, \eqref{eqhne}  we find that 
 \begin{multline}\label{414}\iint_{\R^{\d+\k}\times \R^{\d+\k}} 
 \( \hat\psi(x)-\hat \psi(y)\)\cdot  \nab \g(x-y) d(\sum_{i=1}^N \delta_{x_i}^{(\eta)} - N\mu\drd)(x)d (\sum_{i=1}^N \delta_{x_i}^{(\eta)} - N \mu\drd)(y)\\
 =  \frac{1}{\c}\int_{\R^{\d+\k}} \nab \hat \psi : [H_{N,\vec{\eta}}^\mu, H_{N,\vec{\eta}}^\mu]  .\end{multline}
 \\
 
{\bf Step 2}: {\it analysis of the diagonal terms.}
The main point is to understand how they vary with $\eta_i$.
Let $\vec{\alpha}$ be such that $\alpha_i \ge \eta_i$ for every $i$.

We may write  that 
\begin{multline}\label{m3}
 \iint_{\R^{\d+\k}\times \R^{\d+\k}} 
 \( \hat\psi(x)-\hat \psi(y)\)\cdot  \nab \g(x-y) d \delta_{x_i}^{(\eta_i)} (x) d\delta_{x_i}^{(\eta_i)} (y)
 \\-    \iint_{\R^{\d+\k}\times \R^{\d+\k}} 
 \( \hat\psi(x)-\hat \psi(y)\)\cdot  \nab \g(x-y)d  \delta_{x_i}^{(\alpha_i)} (x)d\delta_{x_i}^{(\alpha_i)} (y)
\\=  \iint_{\R^{\d+\k}\times \R^{\d+\k}}  (\hat\psi(x)-\hat\psi(y)) \cdot \nab \g(x-y) d(\delta_{x_i}^{(\eta_i)}- \delta_{x_i}^{(\alpha_i)}) (x) d(\delta_{x_i}^{(\eta_i)}- \delta_{x_i}^{(\alpha_i)}) (y)
\\+ 2 \iint_{\R^{\d+\k}\times \R^{\d+\k}}  (\hat\psi(x)-\hat\psi(y)) \cdot \nab \g(x-y) d \delta_{x_i}^{(\alpha_i)}(x) d (\delta_{x_i}^{(\eta_i)}- \delta_{x_i}^{(\alpha_i)}) (y)
 .\end{multline}
 We claim that
 \be\label{stv} \iint_{\R^{\d+\k}\times \R^{\d+\k}}  (\hat\psi(x)-\hat\psi(y)) \cdot \nab \g(x-y) d \delta_{x_i}^{(\alpha_i)}(x) d (\delta_{x_i}^{(\eta_i)}- \delta_{x_i}^{(\alpha_i)}) (y)=0.\ee
 Assuming this, 
 inserting  it to \eqref{m3} and using \eqref{pfa} and  Lemma \ref{lem44},  we conclude that 
 \begin{multline}\label{m4}
 \iint_{\R^{\d+\k}\times \R^{\d+\k}} 
 \( \hat\psi(x)- \hat\psi(y)\)\cdot  \nab \g(x-y)d  \delta_{x_i}^{(\eta)} (x)d\delta_{x_i}^{(\eta)} (y)
\\ -    \iint_{\R^{\d+\k}\times \R^{\d+\k}} 
 \( \hat\psi(x)-\hat \psi(y)\)\cdot  \nab \g(x-y) d \delta_{x_i}^{(\alpha_i)} (x)d\delta_{x_i}^{(\alpha_i)} (y)\\
 = \frac{1}{\c} \int_{\R^{\d+\k}}  \nab \hat \psi : [\f_{\alpha_i, \eta_i}(\cdot -x_i), \f_{\alpha_i, \eta_i} (\cdot -x_i)].\end{multline}
 \\
 
 {\bf Step 3}: {\it proof of \eqref{stv}}.
 Let us write the quantity in \eqref{stv}  as 
 \begin{multline}\label{ml2} 2 \iint_{\R^{\d+\k}\times \R^{\d+\k}}  (\hat\psi(x)-\hat\psi(y)) \cdot \nab \g(x-y) d\delta_{x_i}^{(2\alpha_i)}(x) d(\delta_{x_i}^{(\eta_i)}- \delta_{x_i}^{(\alpha_i)}) (y)
 \\+2 \iint_{\R^{\d+\k}\times \R^{\d+\k}}  (\hat\psi(x)-\hat\psi(y)) \cdot \nab \g(x-y)d\(\delta_{x_i}^{(\alpha_i)}- \delta_{x_i}^{(2\alpha_i)} \)(x) d\(\delta_{x_i}^{(\eta_i)}- \delta_{x_i}^{(\alpha_i)}  \) (y).\end{multline}
In view of \eqref{defde}, $\nab\g * \delta_{x_i}^{(2\alpha_i)} = \nab \g_{2\alpha_i}(\cdot -x_i) $ and in view of   \eqref{pfa}, $\nab \g *  (\delta_{x_i}^{(\eta_i)}-\delta_{x_i}^{(\alpha_i )} )= \nab 
 \f_{\alpha_i, \eta_i}(x-x_i)$. 
 We may thus rewrite the first term in \eqref{ml2} as 
 $$2 P.V. \int_{\R^{\d+\k} } \hat\psi \cdot \nab \f_{\alpha_i, \eta_i}(\cdot -x_i) d\delta_{x_i}^{(2\alpha_i)} + 2P.V.  \int_{\R^{\d+\k} }\hat \psi \cdot   \nab \g_{2\alpha_i}(\cdot -x_i) d \( \delta_{x_i}^{(\eta_i)}- \delta_{x_i}^{(\alpha_i)}  \) .$$
 But $ \nab \f_{\alpha_i, \eta_i}(\cdot -x_i) $
 is supported in $B(x_i, \alpha_i)$ while  $\delta_{x_i}^{(2\alpha_i)}$ is supported on $\partial B(x_i, 2\alpha_i)$, and in the same way $\nab \g_{\alpha_i}(\cdot -x_i)$ vanishes in $B(x_i, 2\alpha_i)$  where $ \delta_{x_i}^{(\eta_i)}- \delta_{x_i}^{(\alpha_i)}  $ is supported, so we conclude that the first term in \eqref{ml2} is zero.
The second term in \eqref{ml2} is equal by \eqref{pfa} and Lemma \ref{lem44}  to 
$$ \frac{1}{\c} \int_{\R^{\d+\k}} \nab \hat \psi : [\f_{2 \alpha_i, \alpha_i} (\cdot -x_i), \f_{\alpha_i , \eta_i} (\cdot -x_i)]  $$ and it is zero, since 
 $\f_{2 \alpha_i, \alpha_i} $ and $\f_{\alpha_i, \eta_i}$ have disjoint supports.  This finishes the proof of \eqref{stv}.
 \smallskip
 
 {\bf Step 4:} {\it combining \eqref{414} and \eqref{m4}}.
 The following lemma allows to recombine the terms obtained at different values of $\eta_i$ while making only a small error.
\begin{lem}\label{lem36}Assume that 
$\mu \in C^{\sigma}(\R^{\d})$ with $\sigma>\s-\d+1$  if $\s \ge \d-1$. Assume 
$\mu \in L^\infty(\R^{\d}) $ or $\mu \in C^\sigma(\R^{\d})$ with $\sigma>0$  if $\s <\d-1$.
If for each $i$  we have  $\eta_i<\alpha_i\le \rr_i$, then
$$
\int_{\R^{\d+\k}}\nab \hat \psi : [H_{N,\vec{\eta}}^\mu,H_{N,\vec{\eta}}^\mu]
=  \int_{\R^{\d+\k}}\nab \hat \psi :  [H_{N,\vec{\alpha}}^\mu,H_{N,\vec{\alpha}}^\mu] +\sum_{i=1}^N \int_{\R^{\d+\k}} \nab \hat \psi: [\f_{\alpha_i, \eta_i}(\cdot-x_i) , \f_{\alpha_i, \eta_i}(\cdot-x_i) ] + \mathcal E$$
 with
 \begin{multline}\label{E}|\mathcal E|\le  C\|\nab \psi\|_{L^\infty}\(   F_N(\XN, \mu)+ \(\frac{N}{\d}\log N\) \indic_{\eqref{glog}} +  (1+ \|\mu\|_{L^\infty}) N^{2-\frac{\d-\s}{\d(\d+1)}}\)
\\ + CN \min \(   \|\psi\|_{L^\infty}\|\mu\|_{L^\infty}   \sum_{i=1}^N \alpha_i^{\d-\s-1 }+  \|\nab\psi\|_{L^\infty}\|\mu\|_{L^\infty} \sum_{i=1}^N \alpha_i^{\d-\s},    \|\psi\|_{W^{1,\infty}}\| \mu\|_{C^{\sigma}}  \sum_{i=1}^N \alpha_i^{\d-\s+\sigma-1}\)
\\+ CN\begin{cases} 
  \|\nab\psi\|_{L^\infty} (1+\|\mu\|_{C^{\sigma}}) \sum_{i=1}^N \alpha_i   \qquad \text{if} \ \s \ge \d-1\\
  \|\nab\psi\|_{L^\infty} (1+\|\mu\|_{L^\infty}) \sum_{i=1}^N \alpha_i   \qquad \text{if} \ \s < \d-1\end{cases}
  \end{multline}where $C$ depends only on $\s, \d$.
\end{lem}
Assuming this, and combining \eqref{45},  \eqref{414} and \eqref{m4} we find that for any $\alpha_i \le \rr_i$,
\begin{multline*}
\iint_{\triangle^c} \( \psi(x)- \psi(y)\)\cdot  \nab \g(x-y) d(\sum_{i=1}^N \delta_{x_i}  - N\mu)(x) d(\sum_{i=1}^N \delta_{x_i} - N \mu)(y) = \frac{1}{\c}\int_{\R^{\d+\k}} \nab\hat \psi : [H_{N,\vec{\alpha}}^\mu,H_{N,\vec{\alpha}}^\mu] \\ - \sum_{i=1}^N \iint_{\R^{\d+\k}\times \R^{\d+\k}} 
\(\hat \psi(x)-\hat\psi(y)\) \cdot \nab \g(x-y) d\delta_{x_i}^{(\alpha_i)} (x) d\delta_{x_i}^{(\alpha_i)}(y)
+ O(\mathcal E)
\end{multline*}
where $\mathcal E$ is as in \eqref{E}.
Using the Lipschitz character of $\psi$ and the expression of $\g$, we find that the second term on the right-hand side can be bounded by \footnote{In the case \eqref{glog} we bound instead $|x-y| |\nab \g(x-y)|$ by $1$, which yields an even better control.}
\begin{multline*}C\|\nab\psi\|_{L^\infty} \sum_{i=1}^N    \iint_{\R^{\d+\k}\times \R^{\d+\k}} \g(x-y) d\delta_{x_i}^{(\alpha_i)} (x)d \delta_{x_i}^{(\alpha_i)}(y)\\=C \|\nab\psi\|_{L^\infty}  \sum_{i=1}^N\int_{\R^{\d+\k}} \g_{\alpha_i}(\cdot -x_i)d \delta_{x_i}^{(\alpha_i)}=  C\|\nab\psi\|_{L^\infty}  \sum_{i=1}^N\g(\alpha_i)  \end{multline*} where we have used \eqref{defde}.
Choosing finally $\alpha_i= \rr_i \le N^{-1/\d}$, bounding pointwise $[H_{N,\vecr}^\mu,H_{N,\vecr}^\mu]$ by $2\yg |\nab H_{N,\vecr}^\mu|^2$ and using \eqref{bornehnr}, while using  \eqref{bgr} to bound $\sum_{i=1}^N \g(\rr_i)$,  we conclude the proof of Proposition \ref{32}.

\begin{proof}[Proof of Lemma \ref{lem36}]
First, 
we observe from \eqref{HHf} that  $[H_{N,\vec{\alpha}}^\mu,H_{N,\vec{\alpha}}^\mu]$ and $[H_{N,\vec{\eta}}^\mu,H_{N,\vec{\eta}}^\mu]
$ only differ in the balls $B(x_i,\alpha_i)$ which are disjoint since $\alpha_i\le \rr_i$, and that  in  each $B(x_i,\alpha_i)$ we have 
$$ H_{N,\vec{\eta}}^\mu=  H_{N,\vec{\alpha}}^\mu + \f_{\alpha_i ,\eta_i}(\cdot - x_i).$$  

We thus deduce that 
\begin{multline}\label{320}\int_{B(x_i,\alpha_i)}  \nab\hat\psi:\( [H_{N,\vec{\eta}}^\mu,H_{N,\vec{\eta}}^\mu]- [H_{N,\vec{\alpha}}^\mu,H_{N,\vec{\alpha}}^\mu] \)\\
= \int_{B(x_i,\alpha_i)}
\nab \hat\psi: \([\f_{\alpha_i, \eta_i}, \f_{\alpha_i, \eta_i}](\cdot- x_i)
+2[\f_{\alpha_i, \eta_i}(\cdot -x_i), H_{N,\vec{\alpha}}^\mu] \)\\
=   \int_{\R^{\d+\k}} 
\nab \hat\psi: \([\f_{\alpha_i, \eta_i}, \f_{\alpha_i, \eta_i}](\cdot- x_i)
+2[\f_{\alpha_i, \eta_i}(\cdot- x_i), H_{N,\vec{\alpha}}^\mu]\).\end{multline} There only remains to control the second part of the right-hand side.
By Lemma \ref{lem44}, we have \begin{multline*}  \int_{\R^{\d+\k}} 
\nab \hat\psi: [\f_{\alpha_i, \eta_i}(\cdot- x_i), H_{N,\vec{\alpha}}^\mu ]\\
=\c \iint_{\R^{\d+\k}\times \R^{\d+\k}}  (\hat\psi (x)-\hat \psi(y)) \cdot \nab \g(x-y) d \( \sum_{j=1}^N \delta_{x_j}^{(\alpha_j)}- N \mu\drd\) (x)d \( \delta_{x_i}^{(\eta_i)}-\delta_{x_i}^{(\alpha_i)}\)(y).
\end{multline*}
In view of \eqref{stv}, we just need  to bound the sum over $i$ of
 \begin{multline}\label{eqp} 
\c \iint_{\R^{\d+\k}\times \R^{\d+\k}}  (\hat\psi (x)-\hat \psi(y)) \cdot \nab \g(x-y) d \(  \sum_{j:j\neq i}   \delta_{x_j}^{(\alpha_j)}- N \mu \drd\) (x)d \( \delta_{x_i}^{(\eta_i)}-\delta_{x_i}^{(\alpha_i)}\)(y)
\\
=  \c \int_{\R^{\d+\k}}\hat \psi\cdot \nab \f_{\alpha_i, \eta_i}(\cdot -x_i)d\(  \sum_{j:j\neq i}   \delta_{x_j}^{(\alpha_j)}-N \mu \drd\)    \\
+ \c \int_{\R^{\d+\k} }\hat \psi \cdot \( \sum_{j: j\neq i} \nab \g_{\alpha_j}(x-x_j) -  N\nab \mathsf{h}^\mu \) d\( \delta_{x_i}^{(\eta_i)}-\delta_{x_i}^{(\alpha_i)}\),
\end{multline} where we used \eqref{pfa}.
\smallskip

{\bf Step 1:} {\it  first term in \eqref{eqp}. } Since $\f_{\alpha_i , \eta_i} (\cdot -x_i)   $ is supported in  $B(x_i, \alpha_i)$, $\delta_{x_j}^{(\alpha_j)}$ in $B(x_j, \alpha_j)$ and the balls are disjoint, one type of terms vanishes and there remains 
$$- N\c \int_{\R^{\d}} \psi\cdot \nab \f_{\alpha_i, \eta_i}(\cdot -x_i)d \mu .$$
Thanks to the explicit form of $\f_{\alpha, \eta}$ we have
$$\f_{\alpha_i, \eta_i}(\cdot -x_i)=\begin{cases}
 \g(x-x_i) - \g(\alpha_i) & \quad \text{for \ } \eta_i<|x-x_i|\le\alpha_i \\   \g(\eta_i)-\g(\alpha_i) &\quad \text{for \ }
|x-x_i|\le \eta_i \end{cases}$$ 
and 
$$\nab \f_{\alpha_i, \eta_i}(\cdot -x_i)= \nab \g(x-x_i) \indic_{ \eta_i\le |x-x_i|\le\alpha_i}.$$
It follows that  
\be \label{intf}  \int_{\R^{\d}} |\f_\alpha|\le C \alpha^{\d-\s}, \qquad \int_{\R^{\d}} |\f_{\alpha_i, \eta_i}|\le C \alpha_i^{\d-\s}, \qquad \int_{\R^{\d}} |\nab \f_{\alpha_i, \eta_i}|\le  C\alpha_i^{\d-\s-1}.\ee 
Indeed, 
it suffices to observe that 
\begin{equation}\label{intfeta}\int_{B(0,\eta)} \f_{\alpha}= C\int_0^\alpha (\g(r)-\g(\alpha))r^{\d-1}\, dr= -\frac{C}{\d}\int_0^r \g'(r) r^\d\, dr ,\end{equation}
with an integration by parts.

We may always write 
\begin{multline}\label{c2}
 \left|\int_{\R^{\d}} \psi \cdot \nab \f_{\alpha_i, \eta_i}(\cdot-x_i)  (\mu-\mu(x_i))\right|\le C \|\psi\|_{L^\infty} \|\mu\|_{C^{0,1}} \int_{\eta_i}^{\alpha_i}  \frac{r^{\d  }}{r^{\s+1}}dr\\ \le C  \|\psi\|_{L^\infty} \|\mu\|_{C^{0,1}  }\alpha_i^{\d-\s },
\end{multline}
and, integrating by parts and using \eqref{intf},
\be\label{c3}\left|\int_{\R^{\d}} \psi \cdot \nab \f_{\alpha_i, \eta_i} (\cdot-x_i) \mu(x_i)\right|\le \|\mu\|_{L^\infty}\|\nab\psi\|_{L^\infty} \int_{\R^{\d}}|\f_{\alpha_i, \eta_i}|\le \|\mu\|_{L^\infty} \|\nab\psi\|_{L^\infty} \alpha_i^{\d-\s}.\ee 
Alternatively, we may use the simpler bound derived from \eqref{intf},
\be \label{c3b} \left|\int_{\R^{\d}} \psi \cdot \nab \f_{\alpha_i,\eta_i}(\cdot - x_i) d\mu\right|\le  C \|\psi\|_{L^\infty} \|\mu\|_{L^\infty} \alpha_i^{\d-\s-1}.\ee
A standard interpolation argument yields that $\|g\|_{(C^{\sigma})^*} \le \|g\|_{(C^{1})^*}^\sigma \|g\|_{(C^0)^*}^{1-\sigma}$ so 
interpolating between \eqref{c2}--\eqref{c3} and \eqref{c3b}, we obtain 
 $$\left|\int_{\R^{\d}} \psi \cdot \nab \f_{\alpha_i, \eta_i}(\cdot-x_i)  d\mu\right|\le  C \|\psi\|_{L^\infty}^{1-\sigma} \| \psi\|_{W^{1,\infty}}^\sigma \|\mu\|_{C^{\sigma}}\alpha_i^{\d-\s+\sigma-1}.$$

We conclude that the sum over $i$ of the first terms in \eqref{eqp} is bounded by both
\be\label{t1}
C
\|\psi\|_{L^\infty}^{1-\sigma} \| \psi\|_{W^{1,\infty}}^\sigma \|\mu\|_{C^{\sigma}}\sum_i \alpha_i^{\d-\s+\sigma-1}
\quad \text{and  }\ 
C \|\psi\|_{L^\infty} \|\mu\|_{L^\infty  }\sum_i\alpha_i^{\d-\s-1} .
\ee

{\bf Step 2:}  {\it second term in   \eqref{eqp}.} We may rewrite the integral as
\begin{multline}\label{su2}
-  N  \int_{\R^{\d+\k}}\hat \psi(x_i) \cdot   \nab_{\R^{\d}}\mathsf{h}^\mu d\( \delta_{x_i}^{(\eta_i)}-\delta_{x_i}^{(\alpha_i)}\)
\\
+\sum_{j: j\neq i}  \int_{\R^{\d+\k}}\hat \psi(x_j) \cdot   \nab \g_{\alpha_j}(x-x_j)  d\( \delta_{x_i}^{(\eta_i)}-\delta_{x_i}^{(\alpha_i)}\)
 \\+ \sum_{j: j\neq i} \int_{\R^{\d+\k}}(\hat\psi-\hat \psi(x_j)) \cdot  \nab \g_{\alpha_j}(x-x_j)  d \( \delta_{x_i}^{(\eta_i)}-\delta_{x_i}^{(\alpha_i)}\)
 \\
+  O
\( N \|\nab\psi\|_{L^\infty} \int_{\R^{\d+\k}} |x-x_i| \left| \nab_{\R^{\d}}{\mathsf{h}}^{\mu}  \right| d\( 
 \delta_{x_i}^{(\eta_i)}+\delta_{x_i}^{(\alpha_i)} \) \),
 \end{multline}
where we used that  the last component of $\hat \psi$ vanishes, so that only the derivatives along the $\R^{\d}$ directions appear.

{\bf Substep 2.1:}  {\it first term of \eqref{su2}.} We may write that 
$\delta_{x_i}^{(\eta_i)}-\delta_{x_i}^{(\alpha_i)}= -\frac{1}{\c} \div (\yg\nab \f_{\alpha_i, \eta_i} (\cdot -x_i))$ and integrate by parts twice to get 
$$\frac{1}{\c}\int_{\R^{\d+\k}}  \yg \nab \(\hat \psi(x_i) \cdot     \nab_{\R^{\d}}\mathsf{h}^\mu\)\cdot \nab \f_{\alpha_i, \eta_i} (\cdot -x_i)= 
\int_{\R^{\d}}   ( \psi (x_i) \cdot \nab \mu )\f_{\alpha_i, \eta_i} .$$
Here, we used  that 
  $-\div (\yg \nab \mathsf{h}^{\mu}) = \c \mu \drd$ and  took the $\hat\psi(x_i) \cdot \nab_{\R^{\d}}$ of this relation.
  In view of \eqref{intf}, this is then bounded by 
$$C\|\psi\|_{L^\infty} \|\nab \mu\|_{L^\infty} \alpha_i^{\d-\s}. $$
Alternatively, we may integrate by parts in $\R^{\d}$  to bound it by 
$$C\|\mu\|_{L^\infty} \( \|\psi\|_{L^\infty} \int_{\R^{\d}}|\nab \f_{\alpha_i, \eta_i}| + \|\nab \psi\|_{L^\infty} \alpha_i^{\d-\s}\).$$
  Interpolating as above, we conclude  with \eqref{intf} that the sum over $i$ of these terms is bounded by 
both  \begin{align}\label{t2} 
& C   \|\mu\|_{C^{\sigma}} \(\|\psi\|_{L^\infty}^\sigma\|\nab \psi\|_{L^\infty}^{1-\sigma} \alpha_i^{\d-\s}+ 
\|\psi\|_{L^ \infty}  \sum_i\alpha_i^{\d-\s+\sigma -\1} \) \\ \nonumber  \text{and} \ 
& \|\mu\|_{L^\infty}\( \|\psi\|_{L^\infty}\sum_i \alpha_i^{\d-\s-1}+ \|\nab \psi\|_{L^\infty} \sum_i \alpha_i^{\d-\s}\)
.\end{align}
  
  {\bf Substep 2.2:} {\it second term of \eqref{su2}.}  Arguing in the same way as for the first term, using that 
  $-\div(\yg \nab \g_{\alpha_j}(\cdot-x_j))=\c \delta_{x_j}^{(\alpha_j)}$ and  the disjointness of the balls, we find that this  term  vanishes.
 
{\bf Substep 2.3:} {\it third term in \eqref{su2}.} We separate the sum into two pieces and   bound this term by 
\begin{multline}\label{mus}
\sum_{j\neq i, |x_i-x_j|\ge N^{-\frac{\ep}{\d}} }\int_{\R^{\d+\k}}(\hat\psi-\hat \psi(x_j)) \cdot \nab \g_{\alpha_j}(\cdot-x_j)  d \( \delta_{x_i}^{(\eta_i)}-\delta_{x_i}^{(\alpha_i)}\) 
\\ + \|\nab \psi\|_{L^\infty} \sum_{j\neq i, |x_i-x_j|\le N^{-\frac\ep\d} }\int_{\R^{\d+\k}}|x-x_j| | \nab \g_{\alpha_j}(x-x_j) | d \( \delta_{x_i}^{(\eta_i)}+\delta_{x_i}^{(\alpha_i)}\) (x),
\end{multline}for some $\ep >0$ to be determined. For the first term of \eqref{mus}, we may use that $|x_i-x_j| \ge N^{-\frac\ep\d}$ to write 
$$\| \nab\( (\hat\psi-\hat \psi(x_j)) \cdot  \nab  \g_{\alpha_j} (x-x_j)\) \|_{L^\infty(B(x_i, \rr_i)) } \le C \|\nab \psi\|_{L^\infty} N^{\frac{\ep(\s+1)}{\d}}$$  and using that $\eta_i\le \alpha_i \le N^{-\frac1\d}$ we may thus bound the sum of such terms by 
$$ C \|\nab \psi\|_{L^\infty} N^{\frac{\ep(\s+1)}{\d}+2-\frac1\d}  .$$Since $|\nab \g_{\alpha}|\le |\nab \g|$, we may bound the second term in \eqref{mus} by 
\be \label{2t}
 \|\nab \psi\|_{L^\infty} \sum_{j\neq i, |x_i-x_j|\le N^{-\frac\ep\d} } |x_i-x_j|^{-\s}.\ee
To bound this, let us  choose $\eta_i =  2 N^{-\frac\ep\d}$ in \eqref{fnmeta2} to obtain that 
$$\sum_{i\neq j} \( \g(x_i-x_j)-\g(2 N^{-\frac\ep\d})\)_+\le F_N(\XN, \mu)+ N\g(2 N^{-\frac\ep\d}) + C \|\mu\|_{L^\infty} N^{ 2 -\frac{\ep(\d-\s) }{\d} }.$$
In the cases  \eqref{formeg}, it  follows that
$$\sum_{i\neq j, |x_i-x_j|\le N^{-\frac\ep\d}} \g(x_i-x_j) \le   C\(  F_N(\XN, \mu) + N^{1+ \frac{(\d-\s)\ep}{\d}} +    \|\mu\|_{L^\infty} N^{2-\frac{\ep(\d-\s)}{\d}} \) .$$
In the cases \eqref{glog}, it follows that 
$$ \sum_{i\neq j, |x_i-x_j|\le N^{-\frac\ep\d}} \log 2  \le F_N(\XN, \mu)+ \frac{N }{\d} \log N+   C(1+ \|\mu\|_{L^\infty} )N,$$
 and this suffices to bound \eqref{2t} as well.
 Choosing $\ep= \frac{1}{\d+1}$,
 we conclude in all cases that the sum over $i$ of the third terms in \eqref{su2} is bounded by 
 \be \label{t3} C\|\nab \psi\|_{L^\infty}\(    F_N(\XN, \mu)+ (1+ \|\mu\|_{L^\infty})  N^{2-\frac{\d-\s}{\d(\d+1)}} +\( \frac{N}{\d}\log N \) \indic_{\eqref{glog}} \).\ee

{\bf Substep 2.4:} {\it  fourth term in \eqref{su2}.}  We may bound it by  $O\(\|\nab\psi\|_{L^\infty} \alpha_i \|\nab_{\R^{\d}} \mathsf{h}^\mu\|_{L^\infty(\R^{\d+\k})}\).$
But since $h^\mu= \g* \mu$, it is straightforward to check that $\|\nab_{\R^{\d}}  \mathsf{h}^\mu\|_{L^\infty(\R^{\d+\k} )}\le \|\nab h^\mu\|_{L^\infty(\R^{\d})}$. 
Using \eqref{hmureg}, we conclude the sum of  these terms is bounded by 
\be\label{t4}
\begin{cases} C\sum_i \alpha_i \|\nab \psi\|_{L^\infty} (1+ \|\mu\|_{L^\infty})\quad \text{if } \s<\d-1\\
C\sum_i \alpha_i \|\nab \psi\|_{L^\infty} (1+ \|\mu\|_{C^{\sigma}})\quad \text{if } \s\ge\d-1 \ \text{and} \ \sigma>\s-\d+1
.\end{cases}\ee

\smallskip

Combining the bounds \eqref{t1}, \eqref{t2}, \eqref{t3}, \eqref{t4}  concludes the proof of the lemma.\end{proof}

\section{Proof of Proposition \ref{prop:monoto}}
\label{app}

We drop the superscripts $\mu$. 
First, $\int_{\R^{\d+\k}}\yg|\nab H_{N, \vec{\eta}}|^2 $ is a convergent integral and  
\be\label{intdoub}\int_{\R^{\d+\k}}\yg |\nab H_{N, \vec{\eta}}|^2=\c\iint_{\R^{\d+\k}\times \R^{\d+\k} } \g(x-y)d\left( \sum_{i=1}^N \delta_{x_i}^{(\eta_i)} -N \mu\drd\right)(x) d\left( \sum_{i=1}^N \delta_{x_i}^{(\eta_i)} -N \mu\drd\right)(y).\ee
Indeed, we may choose $R$ large enough so that all the points of $\XN$ are contained in the ball $B_R = B(0, R)$ in $\R^{\d+\k}$.
    By Green's formula  and \eqref{eqhne}, we have
$$
\int_{B_R}\yg |\nabla H_{N,\vec{\eta}}|^2 \\= \int_{\partial B_R}\yg H_{N,\vec{\eta}} \frac{\partial H_N}{\partial {n}} + \c\int_{B_R} H_{N,\vec{\eta}}\, d \left(   \sum_{i=1}^N \delta_{x_i}^{(\eta_i)}-N \mu\drd\right)  .$$
Since $\int d(\sum_i \delta_{x_i}-N\mu)=0$, the function $H_{N,\vec{\eta}}$ decreases like $1/|x|^{\s+1}$ and $\nab H_{N,\vec{\eta}}$ like $1/|x|^{\s+2}$ as $|x|\to \infty$, 
hence  the boundary integral tends to $0$ as $R \to \infty$, and  we may write 
$$\int_{\R^{\d+\k}}\yg|\nab H_{N, \vec{\eta}} |^2= \c \int_{\R^{\d+\k}} H_{N,\vec{\eta}} \, d \left(   \sum_{i=1}^N \delta_{x_i}^{(\eta_i)}-N \mu\drd \right)  $$ and thus by \eqref{eqhne},  \eqref{intdoub} holds.
We may next write that
\begin{multline*}
\lim_{\eta\to 0 }\Big[ \iint_{\R^{\d+\k}\times \R^{\d+\k} } \g(x-y) d\left( \sum_{i=1}^N \delta_{x_i}^{(\eta_i)} -N \mu\drd\right)(x)d \left( \sum_{i=1}^N \delta_{x_i}^{(\eta_i)} -N \mu\drd\right)(y) 
- \sum_{i=1}^N \g(\eta_i)\Big] 
\\=
\iint_{\triangle^c} \g(x-y)  d \left( \sum_{i=1}^N \delta_{x_i}-N \mu\drd\right)(x)d \left( \sum_{i=1}^N \delta_{x_i} -N \mu\drd\right)(y) 
 \end{multline*}
and we deduce in view of \eqref{intdoub} that \eqref{def:FNbis} holds.
 
 We next turn to the proof of \eqref{fnmeta2}, adapted from \cite{ps}. 
 \def\fae{\f_{\alpha_i,\eta_i}}
\def\faej{\f_{\alpha_j, \eta_j}}
%We then let $\f_{\alpha, \eta}:= \f_{\alpha}-\f_{\eta}$.  We note that  it  vanishes outside $B(0,\eta)$, and  
 %\begin{equation}\label{faee}
%\g(\eta)- \g(\alpha)\le \g_\eta- \g_\alpha = \fae\le 0 \end{equation}  and $\f_{\alpha, \eta}$ solves (cf. \eqref{divf})
%\begin{equation}\label{eqfae}
%-\div (\yg \nab \f_{\alpha, \eta})= \c (\delta_0^{(\eta)}- \delta_0^{(\alpha)}).
%\end{equation}
From \eqref{eqhne} applied with $\vec{\eta}=\vec{\alpha}$ and 
  in view of
\eqref{defHN},  we have
$\nab H_{N,\vec{\eta}}= \nab H_{N, \vec{\alpha}} + \sum_{i=1}^N\nab \fae(\cdot -x_i)$ thus
\begin{multline*}
\int_{\R^{\d+\k}} \yg |\nab H_{N,\vec{\eta}}|^2 = \int_{\R^{\d+\k}} \yg |\nab H_{N,\vec{ \alpha}}|^2 + \sum_{i, j} \int_{\R^{\d+\k}} \yg \nab \fae (x-x_i) \cdot \nab \fae (x-x_j) 
\\ +  2 \sum_{i=1}^N\int_{\R^{\d+\k}} \yg \nab \fae(x-x_i)\cdot \nab H_{N, \vec{\alpha}}.\end{multline*}
Using \eqref{sfe}, we first write
\begin{align*}
&  \sum_{i, j} \int_{\R^{\d+\k}} \yg \nab \fae (x-x_i) \cdot \nab \fae (x-x_j) 
\\ 
& = - \sum_{i,j} \int_{\R^{\d+\k}}  \fae (x-x_i) \div (\yg \nab \fae(x-x_j))= \c\sum_{i,j} \int_{\R^{\d+\k}} \fae (x-x_i) d(  \delta_{x_j}^{(\eta_i)}- \delta_{x_j}^{(\alpha_i)}) .
\end{align*}
Next, using \eqref{eqhne}, we write
\begin{multline*}
2 \sum_{i=1}^N \int_{\R^{\d+\k}} \yg \nab \fae(x-x_i)\cdot \nab H_{N, \vec{\alpha}}
 = -2 \sum_{i=1}^N \int_{\R^{\d+\k}} \fae(x-x_i) \div (\yg \nab H_{N,\vec{\alpha}}) 
\\=  2
\c  \sum_{i=1}^N\int_{\R^{\d+\k}} \fae(x-x_i) \, d\Big( \sum_{j=1}^N \delta_{x_j}^{(\alpha_i)} - \mu\drd\Big) . 
\end{multline*}
These last two equations add up  to give a right-hand side equal to 
\begin{multline}\label{rhs}
\sum_{i\neq j} \c \int_{\R^{\d+\k}} \fae (x-x_i) d(  \delta_{x_j}^{(\alpha_i)}+ \delta_{x_j}^{(\eta_j)}) 
- 2\c \sum_{i=1}^N \int\fae(x-x_i) d \mu\drd
\\+ N\c \int_{\R^{\d+\k}} \fae d( \delta_{0}^{(\alpha_i)} + \delta_0^{(\eta_i)} )\, . \end{multline}
We then note that $\int\fae\, d ( \delta_0^{(\alpha_i)} + \delta_0^{(\eta_i)} ) = -\int \f_{\eta_i}d\delta_0^{(\alpha_i)}=-( \g(\alpha_i)-\g(\eta_i))  $ by definition of $\f_\eta$ and the fact that $\delta_0^{(\alpha)}$ is a measure supported on $\partial B(0, \alpha)$ and of mass $1$.
Secondly, by \eqref{intf} we may bound  $\int_{\R^{\d}}\fae(x-x_i) \mu\drd$ by $  C \|\mu\|_{L^\infty}\alpha_i^{\d-\s}.$

Thirdly, we observe that since $\f_{\alpha_i, \eta_i} \le 0$, the first term in \eqref{rhs} is nonpositive and we may bound it above by 
\begin{multline*}\sum_{i\neq j} \c\int_{\R^{\d+\k}} \f_{\alpha_i, \eta_i} d\delta_{x_j}^{(\alpha_j)}
\le \sum_{i\neq j}\c \int_{\R^{\d+\k}}\( \g_{\eta_i}(x-x_i)- \g_{\alpha_i}(x-x_i) \) d\delta_{x_j}^{(\alpha_j)}
\\ \le \sum_{i\neq j}\c \int_{\R^{\d+\k}}
\( \g(\eta_i) - \g_{\alpha_i}(|x_i-x_j|+\alpha_j) \)_-\end{multline*}
 where we used the fact that $\g_\alpha$ is radial decreasing.
Combining the previous relations yields
\begin{multline*}  
- C  N \|\mu\|_{L^\infty}\sum_{i=1}^N  \eta_i^{\d-\s}+ \c \sum_{i\neq j}\( \g_{\alpha_i}(|x_i-x_j|+\alpha_j) -\g(\eta_i)\)_+
\\ \le \(\int_{\R^{\d+\k}} \yg |\nab H_{N, \vec{\alpha}}|^2  -\c\sum_{i=1}^N  \g(\alpha_i)\) -\( \int_{\R^{\d+\k}}\yg |\nab H_{N,\vec{\eta}}|^2 - \c\sum_{i=1}^N \g(\eta_i)\)\end{multline*}
 and letting all $\alpha_i \to 0$  finishes the proof in view of \eqref{def:FNbis}. 

%\newpage
\appendix

\section{Mean-field limit for monokinetic Vlasov systems, \\with Mitia Duerinckx} \label{app1}
In this appendix, we turn to examining the mean-field limit of solutions of Newton's second-order system of ODEs, that is,
\begin{gather}\label{eq:MFL0}
\begin{cases}
\dot x_{i}=v_{i},\\
\dot v_{i}=-\frac1N\nabla_{x_i}\Hc_N(x_1,\ldots,x_N),\\
x_{i}(0)=x_{i}^0,\quad v_i(0)=v_{i}^0,
\end{cases}\qquad i=1,\ldots,N
\end{gather}
where $\Hc_N$ is the interaction energy defined in~\eqref{1.3}.
We then consider the phase-space empirical measure
\[f_N^t:=\frac1N\sum_{i=1}^N\delta_{(x_i^t,v_i^t)}~\in~\Pc(\R^{\d}\times\R^{\d}),\] where $\Pc$ denotes probability measures,
associated to a solution $Z_N^t:=((x_1^t,v_1^t),\ldots,(x_N^t,v_N^t))$ of the flow~\eqref{eq:MFL0}. If the points $(x_i^0,v_i^0)$, which themselves depend on $N$, are such that $f_N^0$ converges to some regular measure $f^0$, then formal computations indicate that for $t>0$, $f_N^t$ should converge to the solution of the Cauchy problem with initial data $f^0$ for the following Vlasov equation,
\begin{align}\label{eq:Vlasov}
\partial_t f+v\cdot\nabla_xf+(\nabla \g\ast\mu)\cdot\nabla_vf=0,\qquad \mu^t(x):=\int_{\R^{\d}}f^t(x,v)\,dv.
\end{align}
In the case of a smooth interaction kernel $\g$, such a mean-field result was first established in the 1970s by a weak compactness argument~\cite{Neunzert-Wick,Braun-Hepp}, while Dobrushin~\cite{Dobrushin-79} gave the first quantitative proof in $1$-Wasserstein distance.
In recent years much attention has been given to the physically more relevant case of singular Coulomb-like kernels, but only very partial results have been obtained. In~\cite{Hauray-Jabin-07,hj}, exploiting a Gr\"onwall argument on the $\infty$-Wasserstein distance between $f_N$ and $f$, Hauray and Jabin treated all interaction kernels $g$ satisfying $|\nabla g(x)|\le C|x|^{-s}$ and $|\nabla^2 g(x)|\le C|x|^{-s-1}$ for some $s<1$. In~\cite{jabinwang}, Jabin and Wang introduced a new approach, allowing them to treat all interaction kernels with bounded gradient.
The same problem has been addressed in~\cite{Boers-Pickl-16,Lazarovici-16,lp}, leading to results that require a small $N$-dependent cutoff of the interaction kernel.

In this appendix, we show that the method presented in the article allows to unlock the mean-field problem with Coulomb-like interaction in the simpler case of monokinetic data, that is, if there exists a regular velocity field $u^0:\R^\d\to\R^\d$ such that $v_i^0\approx u^0(x_i^0)$ for $i=1,\ldots,N$, which implies that the solutions remain ``monokinetic", at least for short times.
Justifying a complete mean-field result for non-monokinetic solutions in the Coulomb case in dimensions $\d\ge2$ remains one of the main open problems in the field.  It is expected to be rendered difficult by the possibility of concentration and filamentation in both space and velocity variables.

In the monokinetic setting, we focus on the (spatial) empirical measure
\[\mu_N^t:=\frac1N\sum_{i=1}^N\delta_{x_{i}^t}~\in~\Pc(\R^{\d}),\]
associated to a solution $Z_N^t$ of the flow~\eqref{eq:MFL0}.
If the points $x_i^0$ are such that $\mu_N^0$ converges to some regular measure $\mu^0$, then formal computations (see for instance~\cite{linz}) lead to expecting that for $t>0$, $\mu_N^t$ converges to the solution $\mu^t$ of the Cauchy problem with initial data $(\mu^0,u^0)$ of the following monokinetic version of~\eqref{eq:Vlasov}\footnote{Formally, the pair $(\mu,u)$ indeed satisfies the system~\eqref{eq:MFL} if and only if the monokinetic ansatz $f^t(x,v):=\mu^t(x)\delta_{v=u^t(x)}$ satisfies the Vlasov equation~\eqref{eq:Vlasov}.},
\begin{align}\label{eq:MFL}
\partial_t \mu+\div(\mu u)=0,\qquad \partial_t u+u\cdot\nabla u=-\nabla \g\ast\mu.
\end{align}
In the Coulomb case, these equations are known as the (pressureless) Euler-Poisson system.
Since for the second-order system~\eqref{eq:MFL0} the total energy splits into a {potential} and a {kinetic} part,
we naturally introduce a modulated \emph{total} energy taking both parts into account:
for $Z_N:=((x_1,v_1),\ldots,(x_N,v_N))$ and for a couple $(\mu,u)\in\Pc(\R^{\d})\times C(\R^{\d})$, we set
\begin{multline*}
H_N(Z_N,(\mu,u)):=N\sum_{i=1}^N|u(x_i)-v_i|^2\\
+\iint_{\R^{\d}\times\R^{\d}\setminus \triangle}\g(x-y)\,d\Big(\sum_{i=1}^N\delta_{x_i}-N\mu\Big)(x)\,d\Big(\sum_{i=1}^N\delta_{x_i}-N\mu\Big)(y),
\end{multline*}
and we view $H_N(Z_N^t,(\mu^t,u^t))$ as a good notion of distance between $\mu_N^t$ and $\mu^t$ that is adapted to the monokinetic setting. The condition $H_N(Z_N^t,(\mu^t,u^t))=o(N^2)$ indeed entails the weak convergence $\mu_N^t\rightharpoonup  \mu^t$ as in Proposition~\ref{procoer}, while due to the kinetic part it also implies that the flow $Z_N^t$ remains monokinetic, that is, $v_i^t\approx u^t(x_i^t)$.
In parallel with Theorem~\ref{th1}, our main result is a Gronwall inequality on the time-derivative of $H_N(Z_N^t,(\mu^t,u^t))$, which implies a quantitative rate of convergence of $\mu_N^t$ to $\mu^t$ in that metric.

\begin{theo}\label{th:MFL-mono}
Assume that $\g$ is of the form~\eqref{formeg} or \eqref{glog}. Assume that~\eqref{eq:MFL} admits a solution $(\mu,u)$ in $[0,T]\times\R^{\d}$ for some $T>0$, with $\mu\in L^\infty([0,T];\Pc\cap L^\infty(\R^{\d}))$ and $u\in L^\infty([0,T];W^{1,\infty}(\R^{\d})^\d)$. In the case $\s\ge \d-1$, also assume $\mu\in L^\infty([0,T];C^\sigma(\R^{\d}))$ for some $\sigma>\s-\d+1$.
Let $Z_N^t$ solve~\eqref{eq:MFL0}. Then there exist constants $C_1,C_2$ depending only on controlled norms of $(\mu,u)$ and on $\d,\s,\sigma$, and there exists an exponent $\beta<2$ depending only on $\d,\s,\sigma$, such that for every $t\in[0,T]$ we have
\[H_N(Z_N^t,(\mu^t,u^t))\le\Big(H_N(Z_N^0,(\mu^0,u^0))+C_1N^\beta\Big)e^{C_2t}.\]
In particular, if $\mu_N^0\rightharpoonup\mu^0$ and is such that
\[\lim_{N\to\infty}\frac1{N^2}H_N(Z_N^0,\mu^0)=0,\]
then the same is true for every $t\in[0,T]$, hence $\mu_N^t\rightharpoonup \mu^t$ in the weak sense.
\end{theo}

\begin{remark}
We briefly comment on the regularity assumptions on the solution $(\mu,u)$ of the mean-field system~\eqref{eq:MFL} in the Coulomb case.
While global existence of a weak solution is known in some settings~\cite{Chen-Wang-96,Guo-98,Ionescu-Pausader-13},
existence of a smooth solution can in general only hold locally in time due to the possible wave breakdown~\cite{Perthame-90,Engelberg-96}.
The local-in-time existence of a smooth solution has then been established in~\cite{Makino-86,Makino-Ukai-87,Gamblin-93}, ensuring that the above regularity assumptions on $(\mu,u)$ indeed hold locally in time for smooth initial data.
In addition, a critical threshold phenomenon was discovered by Engelberg, Liu, and Tadmor~\cite{Engelberg-Liu-Tadmor-01,Liu-Tadmor-03}: global existence of a smooth solution actually holds whenever the initial data satisfies some critical condition, while finite-time breakdown occurs otherwise. For some initial data, the required regularity assumptions on $(\mu,u)$ may thus even hold globally in time.
\end{remark}

Before turning to the proof of Theorem~\ref{th:MFL-mono}, we start with a weak-strong stability principle for the mean-field system~\eqref{eq:MFL}, analogous to \eqref{wsu} for the mean-field equation~\eqref{limeq} in the first-order case.

\begin{lem}
Let $(\mu_1^t,u_1^t)$ and $(\mu_2^t,u_2^t)$ denote two smooth solutions of~\eqref{eq:MFL} in $[0,T]\times\R^{\d}$. Then there is a constant $C$ depending only on $\d$ such that, in terms of
\[H((\mu_1^t,u_1^t),(\mu_2^t,u_2^t)):=\int_{\R^{\d}}|u_1^t-u_2^t|^2\mu_1^t+\int_{\R^{\d}}\int_{\R^{\d}}\g(x-y)\,d(\mu_1^t-\mu_2^t)(x)\,d(\mu_1^t-\mu_2^t)(y),\]
we have
\[H((\mu_1^t,u_1^t),(\mu_2^t,u_2^t))\le e^{C\int_0^t\|\nabla u_2^u\|_{L^\infty}du}H((\mu_1^0,u_1^0),(\mu_2^0,u_2^0)).\qedhere\]
\end{lem}

%{\color{blue}[[On Žnonce ceci juste pour des solutions rŽgulires (sans prŽciser), alors qu'on voudrait l'Žnoncer pour des solutions $\Ld^\infty$ avec $\nabla u_2\in\Ld^1([0,T];\Ld^\infty(\R^d))$. En effet, sinon, a ne para"t pas trs clair comment donner un sens p.ex.\@ aux termes $\int_{\R^d}\nabla|u_1-u_2|^2\cdot\mu_1 u_1$ qui s'annulent dans le calcul ci-dessous.]]}

\begin{proof}
Let $h_i:=\g\ast\mu_i$.
We compute
\begin{eqnarray*}
\partial_tH((\mu_1,u_1),(\mu_2,u_2))&=&\int_{\R^{\d}} |u_1-u_2|^2\partial_t\mu_1+2\int_{\R^{\d}}(u_1-u_2)\cdot(\partial_tu_1-\partial_tu_2)\,\mu_1\\
&&+2\int_{\R^d}(h_1-h_2)(\partial_t\mu_1-\partial_t\mu_2)\\
&=&\int_{\R^{\d}} \nabla|u_1-u_2|^2\cdot \mu_1u_1-2\int_{\R^{\d}}(u_1-u_2)\cdot(u_1\cdot\nabla u_1-u_2\cdot\nabla u_2)\,\mu_1\\
&&-2\int_{\R^{\d}}\nabla(h_1-h_2)\cdot(u_1-u_2)\,\mu_1+2\int_{\R^{\d}}\nabla(h_1-h_2)\cdot(\mu_1u_1-\mu_2u_2).
\end{eqnarray*}
Decomposing $2(u_1-u_2)\cdot(u_1\cdot\nabla u_1-u_2\cdot\nabla u_2)=u_1\cdot\nabla|u_1-u_2|^2+2\nabla u_2:(u_1-u_2)\otimes(u_1-u_2)$, we find after straightforward simplifications,
\begin{multline*}
{\partial_tH((\mu_1,u_1),(\mu_2,u_2))}\\
=-2\int_{\R^d}\mu_1\nabla u_2:(u_1-u_2)\otimes(u_1-u_2)+2\int_{\R^d}u_2\cdot\nabla(h_1-h_2)(\mu_1-\mu_2).
\end{multline*}
In the Coulomb case, we recognize from~\eqref{divstres} in the second right-hand side term the divergence of the stress-energy tensor $[h_1-h_2,h_1-h_2]$, hence
\begin{eqnarray*}
\lefteqn{\partial_tH((\mu_1,u_1),(\mu_2,u_2))}\\
&= &-2\int_{\R^{\d}}\mu_1\nabla u_2:(u_1-u_2)\otimes(u_1-u_2)-\frac1{\cd}\int_{\R^{\d}} u_2\cdot\div[h_1-h_2,h_1-h_2]\\
& \le &C\|\nabla u_2\|_{L^\infty}\int_{\R^{\d}}|u_1-u_2|^2\mu_1+\frac C{\cd}\|\nabla u_2\|_{L^\infty}\int_{\R^{\d}}|\nabla(h_1-h_2)|^2\\
& = & C\|\nabla u_2\|_{L^\infty}H((\mu_1,u_1),(\mu_2,u_2)),
\end{eqnarray*}
and the result follows by Gronwall's lemma.
In the Riesz case, replacing~\eqref{divstres} by~\eqref{divh} and by~\eqref{stresst}--\eqref{divt}, the same conclusion follows.
\end{proof}

Taking inspiration from the above calculations for the weak-strong stability principle, we now turn to the proof of Theorem~\ref{th:MFL-mono}.

\begin{proof}[Proof of Theorem~\ref{th:MFL-mono}]
Let $h:=\g\ast\mu$.
We decompose
\begin{multline*}
\partial_tH_N(Z_N,(\mu,u))=2N\sum_{i=1}^N(u(x_i)-v_i)\cdot\big((\partial_tu)(x_i)+\dot x_i\cdot\nabla u(x_i)-\dot v_i\big)\\+2\sum_{i\ne j}\dot x_i\cdot\nabla \g(x_i-x_j)+2N^2\int_{\R^d}h\partial_t\mu
-2N\sum_{i=1}^N\int_{\R^{\d}}\dot x_i\cdot\nabla \g(x_i-y)\,d\mu(y)-2N\sum_{i=1}^N\int_{\R^{\d}}\g(x_i-\cdot)\,\partial_t\mu.
\end{multline*}
Inserting equations~\eqref{eq:MFL0} and~\eqref{eq:MFL}, we obtain
\begin{multline*}
\partial_tH_N(Z_N,(\mu,u))\\
=2N\sum_{i=1}^N(u(x_i)-v_i)\cdot\Big(-(u\cdot\nabla u)(x_i)-\nabla h(x_i)+v_i\cdot\nabla u(x_i)+\frac1N\sum_{j:j\ne i}\nabla \g(x_i-x_j)\Big)\\
+2\sum_{i\ne j}v_i\cdot\nabla \g(x_i-x_j)+2N^2\int_{\R^{\d}}\nabla h\cdot u\mu\\
-2N\sum_{i=1}^Nv_i\cdot\nabla h(x_i)-2N\sum_{i=1}^N\int_{\R^{\d}}\nabla \g(\cdot-x_i)\cdot u\,d\mu,
\end{multline*}
and hence, after straightforward simplifications,
\begin{multline*}
\partial_tH_N(Z_N,(\mu,u))=-2N\sum_{i=1}^N\nabla u(x_i):(u(x_i)-v_i)\otimes(u(x_i)-v_i)\\
+N^2\iint_{\R^{\d}\times\R^{\d}\setminus\triangle}(u(x)-u(y))\cdot\nabla \g(x-y)\,d(\mu_N-\mu)(x)\,d(\mu_N-\mu)(y).
\end{multline*}
Applying Proposition~\ref{32} to estimate the second right-hand side term, we are led to
\begin{multline*}
\partial_tH_N(Z_N,(\mu,u))\\
\le C\|\nabla u\|_{L^\infty}\Big(H_N(Z_N,(\mu,u))+(1+\|\mu\|_{L^\infty})N^{1+\frac \s\d}+N^{2-\frac{\s+1}2}+ \Big(\frac{N}{\d} \log N \Big) \indic_{\eqref{glog} }\Big)\\
+C\begin{cases}
(1+\|\mu\|_{C^\sigma})\|\nabla u\|_{L^\infty}N^{2-\frac1\d}+\|u\|_{W^{1,\infty}}\|\mu\|_{C^\sigma}N^{1+\frac{\s+1-\sigma}\d},&\text{if $\s\ge \d-1$,}\\
(1+\|\mu\|_{L^\infty})\|\nabla u\|_{L^\infty}N^{2-\frac1\d}+\|u\|_{L^\infty}\|\mu\|_{L^\infty}N^{1+\frac{\s+1}\d}\\
\hspace{4cm}+\|\nabla u\|_{L^\infty}\|\mu\|_{L^\infty}N^{1+\frac \s\d},&\text{if $\s< \d-1$},
\end{cases}
\end{multline*}
and the conclusion follows by Gronwall's lemma. 
\end{proof}

\vskip .5cm
\noindent
\textsc{Mitia Duerinckx}\\
Universit\'e Libre de Bruxelles, boulevard du Triomphe, 1050 Brussels, Belgium,\\
\& Sorbonne Universit\'e, CNRS, UMR 7598, Laboratoire Jacques-Louis Lions, 4 place Jussieu, 75005 Paris, France.\\
Email: {mduerinc@ulb.ac.be}
\vspace{.2cm}

\noindent
\textsc{Sylvia Serfaty}\\
Courant Institute, New York University, 251 Mercer street, New York, NY 10012, USA.\\
Email: {serfaty@cims.nyu.edu}

\end{document}